\documentclass[12pt]{amsart}

\usepackage{epsfig}
\usepackage{mathtools}
\usepackage[subnum]{cases}
\usepackage{enumerate}
\usepackage{amsmath,amssymb}
\usepackage{color}
\usepackage[]{hyperref}
\usepackage{mathrsfs}
\usepackage{amssymb,url}
\usepackage{cite}
\usepackage{verbatim}
\usepackage{tikz,xstring}
\usepackage{pgfplots}
\pgfplotsset{compat=1.18} 

\allowdisplaybreaks[4]
\usepackage{amsmath, graphicx, amsfonts,amssymb,amssymb,mathrsfs,amsthm,MnSymbol}
\usepackage{mathtools} 

\usepackage{color}

\def\be{\begin{equation}}
	\def\ee{\end{equation}}
\def\bea{\begin{eqnarray}}
	\def\eea{\end{eqnarray}}
\def\bt{\begin{theorem}}
	\def\et{\end{theorem}}
\def\bl{\begin{lemma}}
	\def\el{\end{lemma}}
\def\br{\begin{remark}}
	\def\er{\end{remark}}
\def\bc{\begin{corollary}}
	\def\ec{\end{corollary}}
\def\bd{\begin{definition}}
	\def\ed{\end{definition}}

\def\b{\beta}

\def\z{\zeta}
\def\k{\kappa}

\def\x{\xi}

\def\cK{\mathcal{K}}

\def\bbR{\mathbb{R}}

\def\b1{B_{1}^}
\def\bp{B_p^}

\def\ba{\begin{array}}
	\def\ea{\end{array}}
\def\ben{\begin{enumerate}}
	\def\een{\end{enumerate}}
\newtheorem{theorem}{Theorem}[section]
\newtheorem{lemma}{Lemma}[section]
\newtheorem{remark}{Remark}[section]
\newtheorem{proposition}[theorem]{Proposition}
\newtheorem{corollary}{Corollary}[section]

\newtheorem{definition}{Definition}[section]
\newtheorem{defn}[theorem]{Definition}

\def\bt{\begin{theorem}}
	\def\et{\end{theorem}}
\def\bl{\begin{lemma}}
	\def\el{\end{lemma}}
\def\br{\begin{remark}}
	\def\er{\end{remark}}
\def\bc{\begin{corollary}}
	\def\ec{\end{corollary}}
\def\bd{\begin{defn}}
	\def\ed{\end{defn}}
\def\bp{\begin{proposition}}
	\def\ep{\end{proposition}}

\makeatletter
\DeclareRobustCommand\widecheck[1]{{\mathpalette\@widecheck{#1}}}
\def\@widecheck#1#2{%
	\setbox\z@\hbox{\m@th$#1#2$}%
	\setbox\tw@\hbox{\m@th$#1%
		\widehat{%
			\vrule\@width\z@\@height\ht\z@
			\vrule\@height\z@\@width\wd\z@}$}%
	\dp\tw@-\ht\z@
	\@tempdima\ht\z@ \advance\@tempdima2\ht\tw@ \divide\@tempdima\thr@@
	\setbox\tw@\hbox{%
		\raise\@tempdima\hbox{\scalebox{1}[-1]{\lower\@tempdima\box
				\tw@}}}%
	{\ooalign{\box\tw@ \cr \box\z@}}}
\makeatother

\newcommand{\R}{\mathbb{R}} 
\newcommand{\Sp}{\mathbb{S}^{n-1}}

\DeclareMathOperator{\conv}{conv}

\DeclareMathOperator{\as}{as}
\DeclareMathOperator{\intt}{int}
\DeclareMathOperator{\hess}{Hess}
\DeclareMathOperator{\relint}{relint}

\DeclareMathOperator{\vol}{vol}
\DeclareMathOperator{\dom}{dom}

\newcommand{\fconc}{{\mbox{\rm Conc}(0,\infty)}}
\newcommand{\fconv}{{\mbox{\rm Conv}(0,\infty)}}

\newcommand*{\medcapplus}{\mathbin{\scalebox{1.25}{\ensuremath{\capplus}}}}%

\begin{document}
	\title[New fiber and graph combinations of convex bodies]{New fiber and graph combinations of convex bodies}
	
	\author{Steven Hoehner and Sudan Xing}
	\date{\today}
	
	
	\begin{abstract} 
	Three new combinations of convex bodies are introduced and studied: the $L_p$ fiber,  $L_p$ chord and  graph combinations. These combinations are defined in terms of the fibers and graphs of pairs of convex bodies, and each operation generalizes the classical Steiner symmetral, albeit in different ways. For the $L_p$ fiber and $L_p$ chord combinations, we derive Brunn--Minkowski-type inequalities and the corresponding Minkowski's first inequalities. We also prove that the general affine surface areas are concave (respectively, convex) with respect to the graph sum, thereby generalizing fundamental results of Ye ({\it Indiana Univ. Math. J.}, 2014) on the monotonicity of the general affine surface areas under Steiner symmetrization. As an application, we  deduce a corresponding Minkowski's first inequality for the $L_p$ affine surface area of a graph combination of convex bodies. 
	\end{abstract}
	
		\maketitle
	\section{Introduction}

 We shall work in $n$-dimensional Euclidean space $\R^n$ with inner product $\langle x,y\rangle=\sum_{i=1}^n x_i y_i$ and norm $\|x\|=\sqrt{\langle x,x\rangle}$, where $x,y\in\R^n$. The standard orthonormal basis of $\R^n$ is $\{e_1,\ldots,e_n\}$, and the origin of $\R^n$ is denoted by $o$. The Euclidean unit ball centered at the origin is  $B_2^n=\{x\in\R^n:\,\|x\|\leq 1\}$. Its boundary, the unit sphere in $\R^n$ centered at the origin, is denoted by $\Sp=\{x\in\R^n:\,\|x\|=1\}$. For a set $A\subset\R^n$, we let $\intt(A),\relint(A)$ and $\partial A$ denote the interior, relative interior, and boundary of $A$, respectively.
	
	A \emph{convex body} is a convex, compact subset of $\R^n$ with nonempty interior.  Throughout the paper, we let $\cK^n$  denote the class of convex bodies in $\R^n$, and we let $\cK_o^n=\{K\in\cK^n:\, o\in\intt(K)\}$. For $K_1,K_2\in\cK^n$, the Hausdorff distance $d_H(K_1,K_2)$ may be defined by
    \[
    d_H(K_1,K_2) = \max\left\{\sup_{x\in K_1}d(x,K_2),\,\sup_{y\in K_2}d(y,K_1)\right\}
    \]
    where for $x\in\R^n$ and a closed set $C\subset\R^n$, $d(x,C)=\inf_{y\in C}\|x-y\|$. The $n$-dimensional volume of $K\in\cK^n$ is denoted $\vol_n(K)$.
    
    For $K\in\cK^n$, the orthogonal projection of $K$  into the hyperplane $u^\perp=\{x\in\R^n:\,\langle x,u\rangle=0\}$ is denoted by $K|u^\perp$. The support function of a convex body $K$ in $\R^n$ is $h_K(u)=\max_{x\in K}\langle x,u\rangle$. For fixed $u\in\Sp$, we let $f_u: K|u^\perp \rightarrow \bbR$ and $g_u: K|u^\perp \rightarrow \bbR$ denote the
	{\it overgraph} and {\it undergraph} functions of $K$ with respect to $u^\perp$, respectively, which are defined by
	\begin{equation*}  
		f_u(x') =\max\{t\in \bbR: (x',t)\in (K|u^\perp)\times\R\} \quad\text{and}\quad
		g_u(x') =-\min\{t\in \bbR: (x',t)\in (K|u^\perp)\times \R\}.
	\end{equation*} 
 This means that 
  \[
 K = \left\{(x',t):\,x'\in K|u^\perp,\,-g_u(x')\leq t\leq f_u(x')\right\}.
 \]
For $x'\in u^\perp$, we denote by 
 \begin{align}
     \ell_K(u;x')&:=\vol_1(K\cap (x'+\R u))
 \end{align}
 the length of the chord of $K$ (in the direction of $u$) passing through $x'\in u^\perp$. Note that in terms of the overgraph and undergraph functions of $K$, we have
 \[
\ell_K(u;x') = f_u(x') + g_u(x').
 \]
 
	For a convex body $K$ in $\R^n$ and a direction $u\in\Sp$, the \emph{Steiner symmetral} $S_{u}(K)$ of $K$ with respect to the  hyperplane 
	$u^\perp$ may be defined as follows. Translate each chord of $K$ in the direction of $u$ until $u^\perp$ bisects the chord. The union of all the translated  chords is $S_u(K)$. In terms of the overgraph and undergraph of $K$, the Steiner symmetral $S_u(K)$ may be expressed as
	\begin{equation}\label{symmetrization:en}
		S_u(K)=\left\{(x', t)\in (K|u^\perp)\times\R:  -\frac{1}{2}(f_u(x')+g_u(x'))\leq t\leq
		\frac{1}{2}(f_u(x')+g_u(x'))\right\}.
	\end{equation} 
	
 A key feature of the Steiner symmetrization is that it preserves volume, i.e., $\vol_n(K)=\vol_n(S_u(K))$, which follows by Cavalieri's principle. Furthermore, any convex body can be transformed into a Euclidean ball of the same volume via a sequence of Steiner symmetrizations. These  properties make the Steiner symmetrization an indispensable tool in a wide variety of ``shape optimization" problems. For more background on Steiner symmetrizations, we refer the reader to, for example, Section 10.3 of Schneider's book \cite{Sch}, the article \cite{symm-in-geom} by Bianchi, Gardner and Gronchi and the chapter \cite{symm-chapter} by Bianchi and Gronchi.

	
	Various constructions have been introduced to interpret the Steiner symmetrization as a sum or combination of two objects (see, e.g., \cite{GHW-2013,symm-chapter,symm-in-geom,McMullen1999}). An example is the classical fiber sum of convex bodies, which was defined by McMullen \cite{McMullen1999} as follows. Fix two complementary linear subspaces  $L, M\subset \mathbb{R}^n$. Any vector $x \in \mathbb{R}^n$ can be decomposed  as $x=x'+y$, where $y \in L$ and $x' \in M$; we denote this orthogonal  decomposition by $x=(x', y)$.   For convex bodies $K_1$ and $K_2$ in $\R^n$ and $a,b\in\R$, McMullen \cite{McMullen1999} defined the \emph{fiber combination} $a\circ K_1\medcapplus b\circ K_2=a\circ_{L|M}K_1\medcapplus_{L|M}b\circ_{L|M} K_2$  by
 \begin{equation}\label{fiber-sum}
a\circ K_1\medcapplus_{L|M} b\circ K_2=\{(x',ay_1+by_2):\,(x',y_i)\in K_i\,\text{ for }\,i=1,2\},
 \end{equation}
 where for a convex body $K$ in $\R^n$ and $a\in\R$, the \emph{fiber dilation} $a\circ K=a\circ_{L|M}K$ is defined  by (see  \cite{McMullen1999})
 \[
a\circ K=\{(x',ay):\,(x',y)\in K\}.
 \]
Here we have adapted McMullen's notation to suit our purposes (see also \cite{symm-chapter,symm-in-geom}).	This definition means that vectors are only added in their $L$-components, while the $M$-components are fixed.   By construction, the fiber of the sum $K_1\medcapplus_{L|M} K_2$ in the ``direction" $L$ and passing through $x'\in M$  is obtained by summing the corresponding fibers of $K_1$ and $K_2$.  In particular, if  $L=\R^n$ and $M=\{o\}$, then \eqref{fiber-sum} becomes the usual Minkowski addition $K_1+K_2=\{x+y:\, x\in K_1, y\in K_2\}$, and if $L=\{o\}$ and $M=\R^n$, then \eqref{fiber-sum} is the usual intersection $K_1\cap K_2$.  Furthermore, as observed in \cite{McMullen1999}, if $L$ is a line and $M$ is its orthogonal hyperplane, then the fiber combination $(\frac{1}{2}\circ K)\medcapplus((-\frac{1}{2})\circ K)$
	is just the Steiner symmetral of $K$ with respect to $M$. Relative to the  complementary subspaces $L$ and $M$, the fiber sum operation is associative and commutative.  It was also shown in  \cite[Theorem 2.3]{McMullen1999}  that if $K_1$ and $K_2$ are convex, then $K_1\medcapplus K_2$ is also convex.

McMullen proved the following Brunn--Minkowski-type  inequality for fiber combinations of convex bodies \cite{McMullen1999}. Setting $\ell:=\dim L$, the result asserts that the $n$-dimensional volume $\vol_n$ is $\tfrac{1}{\ell}$-concave with respect to the fiber sum of convex bodies. More specifically, McMullen proved that for all convex bodies $K_1,K_2\in\cK^n$ and every $\vartheta\in[0,1]$, it holds that
	\begin{equation}\label{fiber-BMI}
		\vol_n((1-\vartheta)\circ K_1\medcapplus \vartheta\circ K_2)^{1/\ell} \geq (1-\vartheta) \vol_n(K_1)^{1/\ell}+\vartheta\vol_n(K_2)^{1/\ell}.
	\end{equation}
	To our knowledge, the equality cases have not yet been  characterized.  

    In \cite{symm-in-geom}, Bianchi, Gardner and Gronchi introduced and studied an abstract framework for symmetrizations in geometry, which includes the fiber sum as a special case (see also \cite{symm-chapter}). In particular, they studied general properties of these symmetrizations, including convergence of successive symmetrizations. The convergence properties of fiber symmetrizations were also recently studied by Ulivelli \cite{Ulivelli}. Some recent applications of fiber sums can also be found in \cite{HLPRY-1,HLPRY-2}. We also mention the related work \cite{GHW-2013} of Gardner, Hug and Weil, where abstract operations between sets in geometry were studied, as well as their general properties.

    	In this article, we introduce new combinations of convex bodies called the $L_p$ fiber,  $L_p$ chord and graph combinations. Our main results include Brunn--Minkowski-type inequalities for the $L_p$ fiber and $L_p$ chord combinations of convex bodies.  In terms of these operations, we introduce  mixed volumes via variational formulas, and we subsequently derive several isoperimetric-type inequalities.  
        
        These new operations are connected by the fact that they all generalize the classical Steiner symmetral, albeit in different ways.  In particular, we prove that the general $L_\phi$ (respectively $L_\psi$) affine surface areas are concave (respectively, convex) under the graph sum operation, thereby generalizing fundamental results of Ye \cite{Ye2013} on the monotonicity of the general affine surface areas under Steiner symmetrizations.  Consequently, we derive an analogue of Minkowski's first inequality for the $L_p$ affine surface area of a graph sum of convex bodies.

        Therefore, our focus is essentially  on various concavity properties of these three operations and applications thereof, which distinguishes this paper from the recent works \cite{symm-in-geom,GHW-2013,Ulivelli}, where general properties of an abstract framework of symmetrizations were studied. 


The remainder of the paper is essentially divided into thirds. First, in Section \ref{Lp-fiber-sec}, the $L_p$ fiber combination of convex bodies is introduced and studied. Next, Section \ref{graph-section} covers graph combinations, and  finally, Section \ref{Lp-chord-section} is devoted to  $L_p$ chord combinations.

        \section{$L_p$ fiber combinations of convex bodies}\label{Lp-fiber-sec}

        In this section, we define an $L_p$ analogue of McMullen's classical fiber sum. Let $L$ and $M$ be  complementary subspaces of $\R^n$, and let $\varphi_1,\varphi_2:\R^n\to\R$ be two convex functions. McMullen \cite{McMullen1999} defined the \emph{fiber convolute} $\varphi_1\boxplus\varphi_2=\varphi_1\boxplus_{L|M}\varphi_2$ relative to the subspace pair $(L,M)$ by
        \begin{equation}
            (\varphi_1\boxplus\varphi_2)(x',y):=\inf\{\varphi_1(x',y_1)+\varphi_2(x',y_2):\,y_1+y_2=y\}
        \end{equation}
        where again we denote $x=(x',y)$ with $x'\in M$ and $y\in L$. As remarked by McMullen \cite{McMullen1999}, in the case $L=\{o\}$ and $M=\R^n$, we recover the sum $\varphi_1+\varphi_2$, while for $L=\R^n$ and $M=\{o\}$ we recover the infimal convolute
        \[
        (\varphi_1\square\varphi_2)(x) = \inf\{\varphi_1(x)+\varphi_2(y-x):\,y\in\R^n\}.
        \]

        In  \cite[Lemma 8.3]{McMullen1999}, McMullen proved that the domain of the fiber convolute is $\dom(\varphi_1\boxplus_{L|M}\varphi_2)=\dom(\varphi_1)\medcapplus_{L|M}\dom(\varphi_2)$. Furthermore, he proved that the support function of the fiber sum satisfies the property
        \begin{equation}\label{fiber-sum-support-fn}
        h_{K_1\capplus_{L|M}K_2} = h_{K_1}\boxplus_{L^\perp|M^\perp}h_{K_2}.
        \end{equation}

        In the 1960s, Firey  \cite{Firey} extended the classical Minkowski combination of convex bodies to the $L_p$ setting as follows. For $p\geq 1$, $a,b\geq 0$ and $K_1,K_2\in\cK_o^n$, Firey defined the $L_p$ combination $a\cdot_p K_1+_p b\cdot_p K_2$ by the relation
	\begin{equation}\label{Lp-support}
		h_{a\cdot_p K_1+_p b\cdot_p K_2}=\left(a h_{K_1}^p+b h_{K_2}^p\right)^{1/p}.
	\end{equation}
    In particular, Firey also established the $L_p$ Brunn--Minkowski inequality and $L_p$ Minkowski's first inequality in \cite{Firey} (see Subsections \ref{LpBMI-sec} and \ref{LpMFI-sec} below).
    
    Firey's work \cite{Firey} was the starting point for the subsequent $L_p$ Brunn--Minkowski--Firey theory, which made a huge leap forward several decades later with the groundbreaking works \cite{lutwak, Lu1} of Lutwak. Since then, the $L_p$ Brunn--Minkowski--Firey theory has seen tremendous activity in Convex Geometry and Functional Analysis; for some examples, we refer the reader to  \cite{BLYZ-2012,ChenEtAl-2020, HaberlFranz2009, campi2002lp,haberl2008lp, lutwak2002sharp, lutwak2004lp, lutwak2000lp, LYZ,  Putterman-2021, Saraglou-2015,werner2008new}. 
    
   If $L$ and $M$ are complementary orthogonal subspaces of $\R^n$, then $L\oplus M=\R^n$, $M^\perp=L$, $L^\perp=M$ and $L\cap M=\{o\}$. Hence, in view of \eqref{fiber-sum-support-fn} and \eqref{Lp-support}, it is natural to define the following $L_p$ extension of McMullen's fiber combination.
        
	\begin{defn}\label{main-def-Lp-fiber-new}  
Let $L$ and $M$ be orthogonal complementary subspaces of $\R^n$, and let $K_1,K_2\in\cK_o^n$. For $p\in\R\setminus\{0\}$, the \emph{$L_p$ fiber sum} $ K_1\medcapplus_p K_2$ with respect to $(L,M)$ is defined by
    \begin{equation}
        h_{K_1\capplus_p  K_2} :=\left(h_{K_1}^p\boxplus_{L|M} h_{K_2}^p\right)^{\frac{1}{p}},
    \end{equation}
    and the \emph{$L_p$ fiber dilate} $a\circ_p K$ of $K\in\cK_o^n$ is defined by
 \begin{align}\label{1st-dilation}
     a\circ_p K :=a^{1/p}\circ K.
 \end{align}
 For  $a,b\geq 0$ which are not both zero, the \emph{$L_p$ fiber combination} $a\circ_p K_1\medcapplus_p b\circ_p K_2$ is then defined by
 \begin{equation}
     h_{a\circ_p K_1\capplus_p b\circ_p K_2} :=\left(h_{a\circ_p K_1}^p\boxplus_{L|M} h_{b\circ_p K_2}^p\right)^{\frac{1}{p}}.
 \end{equation}
	\end{defn}
In particular, choosing $p=1$,  we recover the classical fiber sum and dilation operations of McMullen in \eqref{fiber-sum-support-fn}; in this case,  we often drop the subscripts and write $\medcapplus=\medcapplus_1$ and $\circ=\circ_1$.  

 Let $\cdot\,|M$ denote the projection  into $M$ in the ``direction" $L$. Observe that if $(K_1|M)\cap(K_2|M)=\varnothing$,   then $K_1\medcapplus_p K_2=\varnothing$. Moreover, if $\relint(K_1|M)\cap\relint(K_2|M)=\varnothing$, then $K_1\medcapplus_p K_2$ may have dimension less than $n$ (and thus empty interior). Therefore, we shall henceforth assume that   $\relint(K_1|M)\cap\relint(K_2|M)\neq\varnothing$ when dealing with the $L_p$ fiber sum $K_1\medcapplus_p K_2$. 

\subsection{Examples of $L_p$ fiber sums} 

Consider the planar convex bodies 
\begin{align*}
K_1&=[-1,1]\times[-1,1]=\{(x_1,x_2)\in\R^2:\, -1\leq x_1,x_2\leq 1\}\\
K_2&=B_2^2=\{(x_1,x_2)\in\R^2:\, x_1^2+x_2^2\leq 1\}. 
\end{align*}
Let $L=\R e_2$ and $M=\R e_1$. Note that $L$ and $M$ are complementary orthogonal subspaces of $\R^2$ and that $\relint(K_1|M)\cap\relint(K_2|M)=(-1,1)$ is nonempty. We have $h_{K_1}(x_1,x_2)=|x_1|+|x_2|$ and $h_{K_2}(x_1,x_2)=\sqrt{x_1^2+x_2^2}$, where $x_1\in M$ and $x_2\in L$. Hence, for any $p\in\R\setminus\{0\}$,
\begin{equation}
\begin{split}
h_{K_1\capplus_p K_2}(x_1,x_2) &= \left(\inf_{z\in \R}\left\{h_{K_1}^p(x_1,z)+h_{K_2}^p(x_1,x_2-z)\right\}\right)^{\frac{1}{p}}\\
&=\left(\inf_{z\in \R}\left\{(|x_1|+|z|)^p+(x_1^2+(x_2-z)^2)^{\frac{p}{2}}\right\}\right)^{\frac{1}{p}}.
\end{split}
\end{equation}

Let $G_{p,x_1,x_2}(z):=(|x_1|+|z|)^p+(x_1^2+(x_2-z)^2)^{\frac{p}{2}}$. For $p=1$, we have $G_{1,x_1,x_2}(z)=|x_1|+|z|+\sqrt{x_1^2+(x_2-z)^2}$. If $x_1=0$, then $G_{1,0,x_2}(z)=|z|+|x_2-z|$, which is the sum of distances from 0 to $z$ and from $x_2$ to $z$. This sum is minimized by any $z\in[0,x_2]$, and hence $\inf_{z\in \R}G_{1,x_1,x_2}(z)=|x_2|$. If $x_1\neq 0$, then a  straightforward computation shows that $\inf_{z\in\R}G_{1,x_1,x_2}(z)=G_{1,x_1,x_2}(0)=|x_1|+\sqrt{x_1^2+x_2^2}$. Therefore, for $p=1$,
\begin{equation}\label{example-support-fn-1}\begin{split}
h_{K_1\capplus_1 K_2}(x_1,x_2)=\inf_{z\in L}G_{1,x_1,x_2}(z)&=\begin{cases}
    |x_2|, &\text{if } x_1=0;\\
    |x_1|+\sqrt{x_1^2+x_2^2}, &\text{if }x_1\neq 0
\end{cases}\\
    &=|x_1|+\sqrt{x_1^2+x_2^2}.
\end{split}
 \end{equation}
 Note that $[-e_1,e_1]+B_2^2$ has the same support function. Since support functions uniquely characterize convex bodies, this implies that $K\medcapplus_1 K_2=[-e_1,e_1]+B_2^2$. The set $K_1\medcapplus_1 K_2$ is illustrated in Figure 1 below.

 For $p=2$, a similar computation shows that 
 \begin{equation}\label{example-support-fn-2}
h_{K_1\capplus_2 K_2}(x_1,x_2) = \begin{cases}
    \sqrt{\frac{1}{2}(|x_1|+|x_2|)^2+x_1^2}, &\text{if }|x_2|\geq |x_1|;\\
    \sqrt{2x_1^2+x_2^2}, &\text{if }|x_2|<|x_1|.
\end{cases}
 \end{equation}

\begin{center}
\begin{tikzpicture}[scale=1.75]
    \draw[->] (-2.5,0) -- (2.5,0) node[right] {$x_1$};
    \draw[->] (0,-1.5) -- (0,1.5) node[above] {$x_2$};

    \node at (0,0) [below left] {$o$};


    \fill[blue!30, opacity=0.5]
        (1, -1) arc (-90:90:1) -- 
        (-1, 1) arc (90:270:1) -- cycle; 

    \draw[blue, thick] (1, -1) arc (-90:90:1); 
    \draw[blue, thick] (-1, 1) arc (90:270:1); 

    \draw[blue, thick] (1, -1) -- (-1, -1);
    \draw[blue, thick] (1, 1) -- (-1, 1);

    \node[blue, font=\normalsize] at (2.4,0.8) {$K_1\medcapplus_1 K_2$};

    \fill (2,0) circle (1pt) node[below right, font=\small] {$(2,0)$};
    \fill (-2,0) circle (1pt) node[below left, font=\small] {$(-2,0)$};
    \fill (0,1) circle (1pt) node[above right, font=\small] {$(0,1)$};
    \fill (0,-1) circle (1pt) node[below right, font=\small] {$(0,-1)$};
\end{tikzpicture}
\end{center}
\noindent{\footnotesize{\bf Figure 1:} The fiber sum $K_1\medcapplus_1 K_2$ of $K_1=[-1,1]\times[-1,1]$ and $K_2=B_2^2$ is shaped like a racetrack.}

    \subsection{Properties of  $L_p$ fiber combinations}

    We begin with a useful characterization of the support function of an $L_p$ fiber dilation.
    
    \begin{lemma}\label{fiber-dilation-support}
        Let $K\in\cK_o^n$, $a>0$ and $p\in\R\setminus\{0\}$. Then
        \[
            h_{a\circ_p K}(x',y)=h_K(x',a^{1/p}y).
        \]
    \end{lemma}

    \begin{proof}
        First, assume that $p=1$. Observe that for any $a>0$, we have $(x',aw)\in a\circ K$ if and only if $(x',w)\in K$ (where $x'\in M$ and $w\in L$). Thus,
        \begin{align*}
            h_{a\circ K}(x',y) &=\sup_{(z',aw)\in a\circ K}\left\langle(z',aw),(x',y)\right\rangle = \sup_{(z',w)\in K}\left\langle(z',aw),(x',y)\right\rangle\\
            &= \sup_{(z',w)\in K}\left\{\langle z',x'\rangle+\langle aw,y\rangle\right\}
            =\sup_{(z',w)\in K}\left\{\langle z',x'\rangle+\langle w,ay\rangle\right\}\\
            &=\sup_{(z',w)\in K}\left\langle(z',w),(x',ay)\right\rangle
            =h_K(x',ay).
        \end{align*}
    The general case for $p\in\R\setminus\{0\}$ now follows from \eqref{1st-dilation} and the preceding computation:
    \[
    h_{a\circ_p K}(x',y)=h_{a^{1/p}\circ K}(x',y)=h_K(x',a^{1/p}y).
    \]
    \end{proof}

    Next, we highlight some basic properties of $L_p$ fiber combinations.

    \begin{proposition}\label{Lp-fiber-props}
        Let $L$ and $M$ be  complementary orthogonal subspaces of $\R^n$, and let $\ell:=\dim L$. Let $p\in\R\setminus \{0\}$. Let $K_1,K_2\in\cK_o^n$ and $a,b\geq 0$. Then: 
       \begin{itemize}
			\item [(i)] (Shadow property)  if $a$ and $b$  are not both 0, we have 
			\[
   (a\circ_p K_1\medcapplus_p b\circ_p K_2)|M=(K_1|M)\cap(K_2|M).
   \]

   \item[(ii)] 	(Commutativity)  It holds that  \[K_1\medcapplus_p K_2=K_2\medcapplus_p K_1.\]
			
			\item [(iii)] (Associativity of sum) if $K_3$ is also a convex body in $\mathcal{K}_o^n$, we have
			$$(K_1\medcapplus_p K_2)\medcapplus_p K_3=K_1\medcapplus_p (K_2\medcapplus_p  K_3).$$ 	
			
			\item [(iv)] (Associativity of dilation) It holds that  \[(ab)\circ_p K=	a\circ_p 	(b\circ_p K).\]
             
            \item [(v)] (Distributivity) We have
            \begin{align*}
                a\circ_p (K_1\medcapplus_p K_2)&=a\circ_p K_1\medcapplus_p a\circ_p K_2.
            \end{align*}
			
			\item[(vi)] (Monotonicity) if $L_1,L_2\in\mathcal{K}_o^n$ are such that $K_1\subset L_1$ and $K_2\subset L_2$, then we have 
			$$a\circ_p K_1\medcapplus_p b\circ_p K_2\subset a\circ_p L_1\medcapplus_p b\circ_p L_2.$$   

		\end{itemize}
    \end{proposition}

    \begin{proof}
    Properties (i) and (ii) follow directly from Definition \ref{main-def-Lp-fiber-new}. To see (iii), note that
    \begin{align*}
        h_{K_1\capplus_p(K_2\capplus_p K_3)}^p(x',y) &=\inf_{y_1+y_2=y}\left\{h_{K_1}^p(x',y_1)+h_{K_2\capplus_p K_3}^p(x',y_2)\right\}\\
        &=\inf_{y_1+y_2=y}\left\{h_{K_1}^p(x',y_1)+\inf_{z_1+z_2=y_2}\left\{h_{K_2}^p(x',z_1)+h_{K_3}^p(x',z_2)\right\}\right\}\\
        &=\inf_{\substack{y_1+y_2=y\\z_1+z_2=y_2}}\left\{h_{K_1}^p(x',y_1)+h_{K_2}^p(x',z_1)+h_{K_3}^p(x',z_2)\right\}\\
        &=\inf_{y_1+z_1+z_2=y}\left\{h_{K_1}^p(x',y_1)+h_{K_2}^p(x',z_1)+h_{K_3}^p(x',z_2)\right\}\\
&=\inf_{\substack{w+z_2=y\\y_1+z_1=w}}\left\{h_{K_1}^p(x',y_1)+h_{K_2}^p(x',z_1)+h_{K_3}^p(x',z_2)\right\}\\
&=\inf_{w+z_2=y}\left\{\inf_{y_1+z_1=w}\left\{h_{K_1}^p(x',y_1)+h_{K_2}^p(x',z_1)\right\}+h_{K_3}^p(x',z_2)\right\}\\
&=\inf_{w+z_2=y}\left\{h_{K_1\capplus_p K_2}^p(x',w)+h_{K_3}^p(x',z_2)\right\}\\
&=h_{(K_1\capplus_p K_2)\capplus_p K_3}^p(x',y).
    \end{align*}
    The identity in (iii) follows.

    Next, observe that the classical fiber dilation is associative:
    \begin{align*}
(ab)\circ K&=\{(x',aby):\,(x',y)\in K\}
=a\circ\{(x',by):\,(x',y)\in K\}=a\circ (b\circ K).
    \end{align*}
Part (iv) thus follows from this fact and \eqref{1st-dilation}. For (v), by Lemma \ref{fiber-dilation-support} we have
\begin{align*}
    h_{a\circ_p(K_1\capplus_p K_2)}(x',y) &=h_{K_1\capplus_p K_2}(x',a^{1/p}y)\\
    &=\left(\inf\left\{h_{K_1}^p(x',y_1)+h_{K_2}^p(x',y_2):\,y_1+y_2=a^{1/p}y\right\}\right)^{1/p}\\
    &=\left(\inf\left\{h_{K_1}^p(x',a^{1/p}y_1)+h_{K_2}^p(x',a^{1/p}y_2):\,y_1+y_2=y\right\}\right)^{1/p}\\
    &=\left(\inf\left\{h_{a\circ_p K_1}^p(x',y_1)+h_{a\circ_p K_2}^p(x',y_2):\,y_1+y_2=y\right\}\right)^{1/p}\\
    &=h_{(a\circ_p K_1)\capplus_p(a\circ_p K_2)}(x',y).
\end{align*}
This implies the desired identity.
  \end{proof}

      \br
	The corresponding properties (i)-(iii) for  the classical fiber combination  were shown by McMullen in \cite{McMullen1999}. 
	\er

The next result states that $n$-dimensional volume is $\frac{\ell}{p}$-homogeneous with respect to the $L_p$ fiber dilation. 

    \begin{lemma}\label{Lp-sum-homogeneity}
        Let $K\in\cK_o^n$, $p\in\R\setminus\{0\}$ and $a\geq 0$. Then
        \[
\vol_n(a\circ_p K) = a^{\ell/p}\vol_n(K).
        \]
    \end{lemma}

    \begin{proof}
        For $x'\in K|M$, let $K(x')=\{t\in L:\,(x',t)\in K\}$ denote the fiber of $K$ over $x'$. Let $\theta(L,M)$ denote the angle between $L$ and $M$. Then
            \begin{align*}
\vol_n(a\circ_p K) &=\theta(L,M)\int_{(a\circ_p K)|M}\vol_\ell\left((a\circ_p K)(x')\right)dx'\\
                &=\theta(L,M)\int_{K|M}\vol_\ell\left(a^{1/p}K(x')\right)dx'\\
                &=a^{\ell/p}\theta(L,M)\int_{K|M}\vol_\ell\left(K(x')\right)dx'\\
                &=a^{\ell/p}\vol_n(K).
            \end{align*}
    \end{proof}

\subsection{The $L_p$ fiber sum preserves convexity for $p\geq 1$}
 
 Next, we show that when $p\geq 1$, the $L_p$ fiber sum is a binary operation on $\cK_o^n$. 

    \begin{theorem}\label{Lp-convexity-lemma}
        For all $K_1,K_2\in\cK_o^n$ and $p\geq 1$, we have $K_1\medcapplus_p K_2\in\cK_o^n$. 
    \end{theorem}

    To prove this, we will need the following 

    \begin{lemma}\label{inf-lowerbd}
        Let $g:\R\to\R$ be a continuous and  increasing function, and let $A$ be a subset of $\R$. If $A$ is bounded below, then $g(\inf A)=\inf g(A)$.
    \end{lemma}

\begin{proof}
    A proof can be found, for example, in \cite{inf-proof}; we include the details for the reader's convenience. Since $A$ is bounded below, $\inf A$ exists. Let $a=\inf A$. On one hand, since $g$ is increasing, for every $x\in A$ we have $g(a)\leq g(x)$, which implies $g(\inf A)\leq \inf g(A)$.  On the other hand, by the definition of infimum, for every $\delta>0$ the set $(a-\delta,a+\delta)\cap A$ is nonempty, which implies that $g((a-\delta,a+\delta)\cap A)$ is nonempty. Since $g$ is continuous, for all $\varepsilon>0$ there exists $\delta>0$ such that 
    \[
    g((a-\delta,a+\delta)\cap A)\subset (g(a)-\varepsilon,g(a)+\varepsilon)\subset g(A). 
    \]
    Therefore, $\inf g(A)\leq g(a)=g(\inf A)$.
\end{proof}

    \begin{proof}[Proof of Theorem \ref{Lp-convexity-lemma}]
        The proof is inspired by Firey's original proof in \cite{Firey} for the $p$-means of convex bodies (see also \cite[Section 9.1]{Sch}). Recall that the support function of a convex body is characterized by the three properties of nonnegativity, positive homogeneity and sublinearity (see, e.g., \cite[Section 1.7]{Sch}). Consider the function defined by
        \[
f(x):=\left(h_{K_1}^p(x)\boxplus_{L|M} h_{K_2}^p(x)\right)^{\frac{1}{p}},\qquad x\in\R^n.
        \]
        We shall show that this function satisfies these three properties and is therefore  the support function of a convex body in $\R^n$.
        
        First, as support functions, $h_{K_1},h_{K_2}\geq 0$, and hence
        \[
f(x)=(h_{K_1}^p(x',y)\boxplus_{L|M}h_{K_2}^p(x',y))^{1/p}=\left(\inf\left\{h_{K_1}(x',y_1)^p+h_{K_2}(x',y_2)^p:\,y_1+y_2=y\right\}\right)^{1/p} \geq 0,
        \]
        where $x=(x',y)$. 
        
        Next, for $\lambda>0$, by the positive homogeneity of the support functions $h_{K_1}$ and $h_{K_2}$, we  have
        \begin{align*}
            f(\lambda x) &=\left(h_{K_1}^p(\lambda x)\boxplus_{L|M} h_{K_2}^p(\lambda x)\right)^{\frac{1}{p}}\\
            &=\left(\inf\left\{h_{K_1}^p(\lambda x',\lambda y_1)+h_{K_2}^p(\lambda x',\lambda y_2):\,y_1+y_2=y\right\}\right)^{1/p}\\
            &=\left(\inf\left\{\lambda^p\left(h_{K_1}^p(x',y_1)+h_{K_2}^p(x',y_2)\right):\,y_1+y_2=y\right\}\right)^{1/p}\\
            &=\lambda \left(\inf\left\{h_{K_1}^p(x',y_1)+h_{K_2}^p(x',y_2):\,y_1+y_2=y\right\}\right)^{1/p}\\
            &=\lambda f(x).
        \end{align*}
        
    It remains to show that $f$ is sublinear, i.e., for all $x_1,x_2\in\R^n$, $f(x_1+x_2)\leq f(x_1)+f(x_2)$. We apply Lemma \ref{inf-lowerbd}   with $g(x)=x^{1/p}$ for $p\geq 1$, use the identity $\inf(A+B)=\inf A+\inf B$, and set $x_1=(x_1',y_1)$ and $x_2=(x_2',y_2)$ to get
    \begin{align*}
       & f(x_1)+f(x_2)\\
       &=h_{K_1\capplus_p K_2}(x_1)+h_{K_1\capplus_p K_2}(x_2)\\
        &=\left(\inf\left\{h_{K_1}^p(x_1',w_1)+h_{K_2}^p(x_1',w_2):\,w_1+w_2=y_1\right\}\right)^{1/p}\\
        &+\left(\inf\left\{h_{K_1}^p(x_2',w_3)+h_{K_2}^p(x_2',w_4):\,w_3+w_4=y_2\right\}\right)^{1/p}\\
        &=\inf\left\{\left(h_{K_1}^p(x_1',w_1)+h_{K_2}^p(x_1',w_2)\right)^{1/p}+\left(h_{K_1}^p(x_2',w_3)+h_{K_2}^p(x_2',w_4)\right)^{1/p}:\,w_1+w_2=y_1,\,w_3+w_4=y_2\right\}.
    \end{align*}
    Here, to apply Lemma \ref{inf-lowerbd}, we have used the fact that the sets  
    \[
    \left\{h_{K_1}^p(x_1',w_1)+h_{K_2}^p(x_1',w_2):\,w_1+w_2=y_1\right\} \quad\text{and}\quad \left\{h_{K_1}^p(x_2',w_3)+h_{K_2}^p(x_2',w_4):\,w_3+w_4=y_2\right\}
    \]
    are bounded below by zero, which follows since $h_{K_i}\geq 0$ for $i=1,2$. 
    Next, by the inequality 
    \[
    \left[(a+a')^p+(b+b')^p\right]^{1/p}\leq (a^p+b^p)^{1/p}+((a')^p+(b')^p)^{1/p} ,
    \]
    which holds for $p>1$ and $a,a',b,b'>0$, along with the  sublinearity of the support functions $h_{K_1}$ and $h_{K_2}$, we get that $f(x_1)+f(x_2)$ is greater than or equal to 
    \begin{align*}
        &\inf\left\{\left[\left(h_{K_1}(x_1',w_1)+h_{K_1}(x_2',w_3)\right)^{p}+\left(h_{K_2}(x_1',w_2)+h_{K_2}(x_2',w_4)\right)^{p}\right]^{1/p}:\,w_1+w_2=y_1,\,w_3+w_4=y_2\right\}\\
        &\geq \inf\left\{\left[h_{K_1}^p(x_1'+x_2',w_1+w_3)+h_{K_2}^p(x_1'+x_2',w_2+w_4)\right]^{1/p}:\,w_1+w_2=y_1,\,w_3+w_4=y_2\right\}\\
        &\geq \inf\left\{\left[h_{K_1}^p(x_1'+x_2',w_1+w_3)+h_{K_2}^p(x_1'+x_2',w_2+w_4)\right]^{1/p}:\,(w_1+w_3)+(w_2+w_4)=y_1+y_2\right\}\\
        &=f(x_1+x_2).
    \end{align*}
    This completes the proof.
    \end{proof}


\subsection{Brunn--Minkowski-type inequality for the $L_p$ fiber combination}\label{LpBMI-sec}
	
	As mentioned earlier, Firey \cite{Firey} established the $L_p$ Brunn--Minkowski--Firey inequality, which asserts that for any $p\geq 1$, the volume functional is $\frac{p}{n}$-concave with respect to $L_p$ addition on $\mathcal{K}_o^n$. More precisely, for all $K_1,K_2\in\cK_o^n$ and every $\vartheta\in[0,1]$, 
	\begin{equation}\label{Lp-BMI}
	\vol_n((1-\vartheta)\cdot_p K_1+_p\vartheta\cdot_p K_2)^{\frac{p}{n}} \geq (1-\vartheta) \vol_n(K_1)^{\frac{p}{n}}+\vartheta\vol_n(K_2)^{\frac{p}{n}}
	\end{equation}
	with equality if and only if $K_1=K_2$. In particular, choosing $p=1$, one recovers the classical Brunn--Minkowski inequality.  

The main result of this section is the following $L_p$ analogue of \eqref{fiber-BMI}.

\begin{theorem}\label{main-BMI}
For all $K_1,K_2\in\cK_o^n$, $\vartheta\in[0,1]$  and $p\geq 1$,
    \begin{equation}
        \vol_n((1-\vartheta)\circ_p K_1\medcapplus_p \vartheta\circ_p K_2)^{\frac{p}{\ell}} \geq \vartheta\vol_n(K_1)^{\frac{p}{\ell}}+(1-\vartheta)\vol_n(K_2)^{\frac{p}{\ell}}. 
    \end{equation}
\end{theorem}
McMullen's fiber Brunn--Minkowski inequality \eqref{fiber-BMI} is the special case $p=1$. We modify Firey's original proof of the   $L_p$ Brunn--Minkowski inequality \cite{Firey}, which is obtained from the classical Brunn--Minkowski inequality. Analogously, Theorem \ref{main-BMI} will be obtained from \eqref{fiber-BMI}.

\begin{proof}
    We may assume that $p>1$.  For $i=1,2$, let $\lambda_i=\vol_n(K_i)^{1/\ell}$, and set
$\vartheta'=\vartheta\lambda_2^p/[(1-\vartheta)\lambda_1^p+\vartheta\lambda_2^p]$. For any $x=(x',y)\in\R^n$, since $g(x)=x^{1/p}$ is continuous and strictly increasing, by Lemma \ref{inf-lowerbd} we obtain
\begin{align*}
    h_{(1-\vartheta')\circ_p\left(\lambda_1^{-1}\circ K_1\right)\capplus_p\vartheta'\circ_p\left(\lambda_2^{-1}\circ K_2\right)}(x',y) 
    &=h_{\frac{1-\vartheta'}{\lambda_1^p}\circ_p K_1\capplus_p\frac{\vartheta'}{\lambda_2^p}\circ_p K_2}(x',y) \\&=\left(\inf\left\{\frac{1-\vartheta'}{\lambda_1^p}h_{K_1}^p(x',y_1)+\frac{\vartheta'}{\lambda_2^p}h_{K_2}^p(x',y_2):\,y_1+y_2=y\right\}\right)^{1/p}\\
    &=\inf\left\{\left[\frac{1-\vartheta'}{\lambda_1^p}h_{K_1}^p(x',y_1)+\frac{\vartheta'}{\lambda_2^p}h_{K_2}^p(x',y_2)\right]^{1/p}:\,y_1+y_2=y\right\}\\
    &=\inf\left\{\left[\frac{(1-\vartheta)h_{K_1}^p(x',y_1)+\vartheta h_{K_2}^p(x',y_2)}{(1-\vartheta)\lambda_1^p+\vartheta\lambda_2^p}\right]^{1/p}:\,y_1+y_2=y\right\}\\
    &=h_{\frac{1-\vartheta}{\nu}\circ_p K_1\capplus_p\frac{\vartheta}{\nu}\circ_p K_2}(x',y),
\end{align*}
where 
\[
\nu:=(1-\vartheta)\lambda_1^p+\vartheta\lambda_2^p=(1-\vartheta)\vol_n(K_1)^{\frac{p}{\ell}}+\vartheta\vol_n(K_2)^{\frac{p}{\ell}}.
\]
Since the support function $h_K$ uniquely determines the convex body $K$, this implies that 
\[
(1-\vartheta')\circ_p\left(\lambda_1^{-1}\circ K_1\right)\medcapplus_p\vartheta'\circ_p\left(\lambda_2^{-1}\circ K_2\right)=\frac{1-\vartheta}{\nu}\circ_p K_1\medcapplus_p\frac{\vartheta}{\nu}\circ_p K_2.
\]
Hence, by Proposition \ref{Lp-fiber-props}(v)  and Lemma \ref{Lp-sum-homogeneity},
\begin{align*}
\vol_n\left((1-\vartheta')\circ_p\left(\lambda_1^{-1}\circ K_1\right)\medcapplus_p\vartheta'\circ_p\left(\lambda_2^{-1}\circ K_2\right)\right)&=\vol_n\left(\frac{1-\vartheta}{\nu}\circ_p K_1\medcapplus_p\frac{\vartheta}{\nu}\circ_p K_2\right)\\
&=\nu^{-\ell/p}\vol_n((1-\vartheta)\circ_p K_1\medcapplus_p\vartheta\circ_p K_2).
\end{align*}
Therefore, it suffices to show that 
\begin{equation}\label{Lp-BMI-wts}
    \vol_n\left((1-\vartheta')\circ_p\left(\lambda_1^{-1}\circ K_1\right)\medcapplus_p\vartheta'\circ_p\left(\lambda_2^{-1}\circ K_2\right)\right) \geq 1.
\end{equation}
For any $p\geq 1$, we have $((1-\vartheta')a^p+\vartheta' b^p)^{1/p}\geq (1-\vartheta')a+\vartheta' b$. Hence, 
\[
(1-\vartheta')\circ_p\left(\lambda_1^{-1}\circ K_1\right)\medcapplus_p\vartheta'\circ_p\left(\lambda_2^{-1}\circ K_2\right) \supset (1-\vartheta')\circ_1\left(\lambda_1^{-1}\circ K_1\right)\medcapplus_1\vartheta'\circ_1\left(\lambda_2^{-1}\circ K_2\right).
\]
Therefore,  by the monotonicity of volume, \eqref{fiber-BMI}, Lemma \ref{Lp-sum-homogeneity}, and the definition of $\lambda_i$, we obtain
\begin{align*}
\nu^{-\ell/p}\vol_n((1-\vartheta)\circ_p K_1\medcapplus_p\vartheta\circ_p K_2) &\geq 
\vol_n\left((1-\vartheta')\circ_1\left(\lambda_1^{-1}\circ K_1\right)\medcapplus_1\vartheta'\circ_1\left(\lambda_2^{-1}\circ K_2\right)\right)\\
&\geq\left[(1-\vartheta')\vol_n(\lambda_1^{-1}\circ K_1)^{1/\ell}+\vartheta'\vol_n(\lambda_2^{-1}\circ K_2)^{1/\ell}\right]^\ell\\
&=1.
\end{align*}
The result follows.
\end{proof}
\subsection{Minkowski's first inequality for $L_p$ fiber combinations}\label{LpMFI-sec}

Let $K_1,K_2\in\cK_o^n$ and $p\geq 1$. In the groundbreaking work \cite{lutwak}, Lutwak  defined the  $L_p$ first mixed volume $V_p(K_1,K_2)$ to be
\[
V_p(K_1,K_2)=\frac{p}{n}\lim_{\varepsilon\to 0^+}\frac{\vol_n(K_1+_p\varepsilon\cdot_p K_2)-\vol_n(K_1)}{\varepsilon}.
\]
Choosing $p=1$ yields the classical first mixed volume. The $L_p$ Minkowski's first inequality, also due to Lutwak \cite{lutwak}, states that
\[
V_p(K_1,K_2)\geq \vol_n(K_1)^{\frac{n-p}{n}}\vol_n(K_2)^{\frac{p}{n}},
\]
with equality if and only if $K_1$ and $K_2$ are dilates.  The case $p=1$ is the classical Minkowski's first inequality. For more background, we refer the reader to, e.g., \cite{lutwak,Sch}.

    Naturally, this leads us to define the following mixed surface area with respect to the $L_p$ fiber sum.

 \bd For complementary orthogonal subspaces $L$ and $M$ of $\R^n$ and $p\in\R\setminus\{0\}$, the    \emph{$L_p$ fiber mixed surface area} $S_p^\capplus(K_1,K_2)$ of $K_1,K_2\in\cK_o^n$ is defined by 
	\begin{equation}\label{fiberpl}
		S_p^\capplus(K_1,K_2) := \lim_{\varepsilon \to 0^+} \frac{	\vol_n(K_1\medcapplus_p\varepsilon \circ_p K_2) -\vol_n(K_1)}{\varepsilon}.
	\end{equation}
    \ed

The next result is the Minkowski's first inequality for the $L_p$ fiber sum.

	\bt\label{Lp-fiberMinkowski1}
		Let $L$ and $M$ be complementary orthogonal subspaces of $\R^n$ with $\dim L=:\ell$, and let $p\geq 1$. For all $K_1,K_2\in\mathcal{K}_o^n$,  
			\begin{equation}\label{desired-ineq}
			S_p^\capplus(K_1, K_2) \geq \frac{\ell}{p}\cdot\vol_n(K_1)^{\frac{\ell-p}{\ell}}\vol_n(K_2)^{\frac{p}{\ell}}.
			\end{equation}
			If, furthermore, $\vol_n(K_1) = \vol_n(K_2)$, then
			\[
			S_p^\capplus(K_1, K_2) \geq \frac{\ell}{p}\cdot\vol_n(K_1).
			\] 
		\et

\begin{proof}
Inequality \eqref{desired-ineq} will be derived from Theorem \ref{main-BMI} by modifying the standard arguments used to prove the classical Minkowski's first inequality (see, e.g., \cite{Gardner-BMI} or \cite{Sch}).  We include the details for the reader's convenience. Substituting $\varepsilon=t/(1-t)$ into \eqref{fiberpl}, and using Proposition \ref{Lp-fiber-props}(v) and Lemma \ref{Lp-sum-homogeneity}, we derive that
            \begin{align*}
                S_p^\capplus(K_1,K_2) &=\lim_{t\to 0^+}\frac{\vol_n\left(\frac{1}{1-t}\circ_p[(1-t)\circ_p K_1\medcapplus_p t\circ_p K_2]\right)-\vol_n(K_1)}{t/(1-t)} \\
                &=\lim_{t\to 0^+}\frac{(1-t)^{-\ell/p}\vol_n\left((1-t)\circ_p K_1\medcapplus_p t\circ_p K_2\right)-\vol_n(K_1)}{t/(1-t)} \\
                 &=\lim_{t\to 0^+}\frac{\vol_n\left((1-t)\circ_p K_1\medcapplus_p t\circ_p K_2\right)-(1-t)^{\ell/p}\vol_n(K_1)}{t(1-t)^{\frac{\ell-p}{p}}} \\
                 &=\lim_{t\to 0^+}\frac{\vol_n((1-t)\circ_p K_1\medcapplus_p t\circ_p K_2)-\vol_n(K_1)}{t}+\lim_{t\to 0^+}\frac{\left[1-(1-t)^{\ell/p}\right]\vol_n(K_1)}{t}\\
                &=\lim_{t\to 0^+}\frac{\vol_n((1-t)\circ_p K_1\medcapplus_p t\circ_p K_2)-\vol_n(K_1)}{t}+\frac{\ell}{p}\cdot\vol_n(K_1).
            \end{align*}
    We now set $f_{p,\ell}(t):=g_{p,\ell}(t)^{p/\ell}$, where $g_{p,\ell}(t):=\vol_n((1-t)\circ_p K_1\medcapplus_p t\circ_p K_2)$.   
    The preceding computation shows that
\begin{align*}
    f_{p,\ell}'(0) &=\frac{d^+}{dt}[g_{p,\ell}(t)^{p/\ell}]_{t=0}\\
    &=\frac{p}{\ell}\cdot\vol_n(K_1)^{\frac{p-\ell}{\ell}}\cdot\lim_{t\to 0^+}\frac{\vol_n((1-t)\circ_p K_1\medcapplus_p t\circ_p K_2)-\vol_n(K_1)}{t}\\
    &=\frac{p}{\ell}\cdot\vol_n(K_1)^{\frac{p-\ell}{\ell}}\left[S_{p}^{\capplus}(K_1,K_2)-\frac{\ell}{p}\cdot \vol_n(K_1)\right]\\
    &=\frac{p}{\ell}\cdot\vol_n(K_1)^{\frac{p-\ell}{\ell}}S_p^\capplus(K_1,K_2)-\vol_n(K_1)^{\frac{p}{\ell}}\\
    &=\frac{pS_p^\capplus(K_1,K_2)-\ell\vol_n(K_1)}{\ell\vol_n(K_1)^{\frac{\ell-p}{\ell}}}.
\end{align*}
Note also that $f_p(0)=\vol_n(K_1)^{p/\ell}$ and $f_p(1)=\vol_n(K_2)^{p/\ell}$. Thus, the desired inequality \eqref{desired-ineq} is equivalent to $f_p'(0)\geq f_p(1)-f_p(0)$.  Now Theorem \ref{main-BMI} says that $f_p(t)$ is concave, so the latter inequality holds true, and hence \eqref{desired-ineq} holds true. 
\end{proof}

\section{Graph combinations of convex bodies}\label{graph-section}

Rather than defining a sum of convex bodies with respect to their fibers,  we may instead define a sum of convex bodies with respect to their overgraphs and undergraphs. Let $u\in\Sp$ and $K\in\cK^n$. If $f$ and $g$ denote the overgraph and undergraph of $K$, respectively, with respect to $u^\perp$, then we write $K\approx(f,-g)_u$. 

\begin{defn}\label{graph-sum-defn}
    Let $u\in\Sp$ and $a,b\in\R$. Given convex bodies $K_1,K_2\in\cK^n$ with  $K_1|u^\perp=K_2|u^\perp$ and $K_i\approx(f_i,-g_i)_u$, $i=1,2$, we define the \emph{graph combination} $a\diamond_u K_1\diamondplus_u b\diamond_u K_2$ by the relation
    \begin{align*}
        a\diamond_u K_1\diamondplus_u b\diamond_u K_2\approx\left(af_1+bf_2,-(ag_1+bg_2)\right)_u.
    \end{align*}
    We also define the \emph{graph dilation} by
    \[
    a\diamond_u K_1 \approx (af_1,-ag_1)_u.
    \]
\end{defn}
Observe that if $K\approx(f,-g)_u$, then $R_u(K)\approx(g,-f)_u$. Hence, $S_u(K)\approx(\tfrac{1}{2}(f+g),-\tfrac{1}{2}(f+g))_u$, which implies that
\begin{equation}
    \frac{1}{2}\diamond_u K\diamondplus_u \frac{1}{2}\diamond_u R_u(K)=S_u(K).
\end{equation}Also note that $1\diamond_u K=K$.

\subsection{Examples of graph sums of convex bodies}

Consider again the planar convex bodies $K_1=[-1,1]\times[-1,1]$ and $K_2=B_2^2$. Let $u=e_2$ and note that $K_1|u^\perp=K_2|u^\perp=[-1,1]$. For $x_1\in[-1,1]$, we have $f_1(x_1)=g_1(x_1)=1$, and  $f_2(x_1)=1$ and $g_2(x_1)=2-x_1^2$. It follows that
\[
K_1\diamondplus_u K_2=\left\{(x_1,x_2)\in\R^2:\, -1\leq x_1\leq 1,\,-\left(1+\sqrt{1-x_1^2}\right)\leq x_2\leq 1+\sqrt{1-x_1^2}\right\}.
\]
Hence, $K_1\diamondplus_u K_2=[-e_2,e_2]+B_2^2$. Recalling that $K_1\medcapplus_1 K_2=[-e_1,e_1]+B_2^2$, we deduce that $K_1\diamondplus_u K_2$ is different from $K_1\medcapplus_1 K_2$. However, in this case, $K_1\diamondplus_u K_2$ and $K_1\medcapplus_1 K_2$ are rotations of each other. For another example in which this is not the case, consider the rectangles $R_1=[-2,2]\times[-1,1]$ and $R_2=[-2,2]\times[-1/2,1/2]$. Then  with $u=e_2$, we have $R_1\diamondplus_u R_2=[-2,2]\times [-3/2,3/2]$ and $R_1\medcapplus_1 R_2=[-4,4]\times [-1/2,1/2]$. The resulting sets are rectangles with different aspect ratios, so they are not rotations or reflections of each other.

Note also that 
\[
h_{K_1\diamondplus_u K_2}(x_1,x_2)=\sup_{(x,y)\in K_1\diamondplus_u K_2}\{xx_1+yx_2\}=\sup_{x\in[-1,1]}\left\{xx_1+(1+\sqrt{1-x^2})x_2\right\}.
\]
Setting $\frac{d}{dx}\left(xx_1+(1+\sqrt{1-x^2})x_2\right)$ equal to 0, we get  $x=x_1/\sqrt{x_1^2+x_2^2}$. Substituting, we obtain $h_{K_1\diamondplus_u K_2}(x_1,x_2)=|x_2|+\sqrt{x_1^2+x_2^2}$. Comparing this with the support function of $K_1\medcapplus_2 K_2$ in \eqref{example-support-fn-2}, we deduce that $K_1\diamondplus_u K_2$ is also different from  $K_1\medcapplus_2 K_2$. Hence, in general, the graph sum and $L_p$ fiber sum of convex bodies are distinct operations on $\cK^2_o$. These examples are illustrated in Figure 2 below.

For another example (which will be used later), let $K_3=\{(x_1,x_2)\in\R^2:\, -1\leq x_1\leq 1,\, x_1^2-1\leq x_2\leq 1\}$. Then $K_3\in\cK^2$ and $K_1|u^\perp=K_3|u^\perp=[-1,1]$. The graph sum of $K_1$ and $K_3$ is
\[
K_1\diamondplus_u K_3=\left\{(x_1,x_2)\in\R^2:\,-1\leq x_1\leq 1,\,x_1^2-2\leq x_2\leq 2\right\}.
\]

\begin{center}
\begin{tikzpicture}
    \draw[->] (-2.5, 0) -- (2.5, 0) node[right] {$x_1$};
    \draw[->] (0, -2.5) -- (0, 2.5) node[above] {$x_2$};


       \fill (1,0) circle (1pt) node[below right, font=\small] {$(1,0)$};
    \fill (-1,0) circle (1pt) node[below left, font=\small] {$(-1,0)$};
    \fill (0,2) circle (1pt) node[above right, font=\small] {$(0,2)$};
    \fill (0,-2) circle (1pt) node[below right, font=\small] {$(0,-2)$};

    \draw[blue, thick, fill=blue!10, opacity=0.7] plot[domain=-1:1, samples=100] (\x, {1 + sqrt(1 - \x*\x)}) -- plot[domain=1:-1, samples=100] (\x, {-(1 + sqrt(1 - \x*\x))}) -- cycle;


    \node[align=center] at (2, 1) {{\color{blue}$K_1\diamondplus_u K_2$}};

\end{tikzpicture}
\begin{tikzpicture}[scale=0.8]
    \draw[->] (-5, 0) -- (5, 0) node[right] {$x_1$};
    \draw[->] (0, -3.5) -- (0, 3.5) node[above] {$x_2$};

    \draw[dashed, black] (2, -3.5) -- (2, 3.5);
    \node[below, black] at (2, -3.5) {\small $2$};
    \draw[dashed, black] (-2, -3.5) -- (-2, 3.5);
    \node[below, black] at (-2, -3.5) {\small $-2$};
    \draw[dashed, black] (4, -3.5) -- (4, 3.5);
    \node[below, black] at (4, -3.5) {$4$};
    \draw[dashed, black] (-4, -3.5) -- (-4, 3.5);
    \node[below, black] at (-4, -3.5) {$-4$};

    \draw[dashed, black] (-5, -1.5) -- (5, -1.5) node[right, black] {\small $-3/2$};
    \draw[dashed, black] (-5, 1.5) -- (5, 1.5) node[right, black] {\small $3/2$};
    \draw[dashed, black] (-5, -0.5) -- (5, -0.5) node[right, black] {\small $-1/2$};
    \draw[dashed, black] (-5, 0.5) -- (5, 0.5) node[right, black] {\small $1/2$};

    \draw[red, thick, fill=red!10, opacity=0.7] (-2, -1.5) rectangle (2, 1.5);
    \node[red, above right] at (1.8, 1.5) {\small $R_1\diamondplus_u R_2$};

    \draw[blue, thick, fill=blue!10, opacity=0.7] (-4, -0.5) rectangle (4, 0.5);
    \node[blue, below right] at (1.8, -0.6) {\small $R_1\medcapplus_1 R_2$};

\end{tikzpicture}
\end{center}
\noindent{\footnotesize{\bf Figure 2:} The graph sum $K_1\diamondplus_u K_2$ is illustrated in the left figure. In the right figure, the sets $R_1\diamondplus_u R_2$ and $R_1\medcapplus_1 R_2$ are depicted.}

\subsection{Properties of graph combinations}

Next, we highlight some  basic properties of  graph combinations of convex bodies.
	
	\begin{proposition}\label{graph-proposition}  Let $u\in\Sp$ and $a,b\geq 0$. Let  $K_1,K_2\in\cK^n$ be such that $K_i\approx(f_i,-g_i)_u$, $i=1,2$, and assume that $K_1|u^\perp=K_2|u^\perp$.

		\begin{itemize}
			\item [(i)] (Shadow property)  We have
			\[
   (a\diamond_u K_1\diamondplus_u b\diamond_u K_2)|u^\perp=K_1|u^\perp=K_2|u^\perp.
   \]

   \item[(ii)] 	(Commutativity)  It holds that  \[K_1\diamondplus_u K_2=K_2\diamondplus_u K_1.\]
			
			\item [(iii)] (Associativity of sum) Let $K_3$ also be a convex body in $\mathcal{K}^n$ with $K_1|u^\perp=K_2|u^\perp=K_3|u^\perp$. Then
			$$(K_1\diamondplus_u K_2)\diamondplus_u K_3=K_1\diamondplus_u (K_2\diamondplus_u  K_3).$$ 	
			
			\item [(iv)] (Associativity of dilation) It holds that  \[ab\diamond_u K=	a\diamond_u	(b\diamond_u K).\]
             
            \item [(v)] (Distributivity) We have
            \begin{align*}
                a\diamond_u (K_1\diamondplus_u K_2)&=a\diamond_u K_1\diamondplus_u a\diamond_u K_2
            \end{align*}
            and
            \[
            (a+b)\diamond_u K_1=a\diamond_u K_1\diamondplus_u b\diamond_u K_1.
            \]
			
			\item[(vi)] (Monotonicity) Let $L_1,L_2\in\mathcal{K}^n$ be such that $K_1\subset L_1$ and $K_2\subset L_2$. Then 
			$$a\diamond_u K_1\diamondplus_u b\diamond_u K_2\subset a\diamond_u L_1\diamondplus_u b\diamond_u L_2.$$   

		\end{itemize}
	\end{proposition}

\begin{proof}
    Part (i) follows immediately from Definition \ref{graph-sum-defn}. For (ii), we have
    \begin{align*}
        K_1\diamond_u K_2\approx (f_1+f_2,-(g_1+g_2))_u=(f_2+f_1,-(g_2+g_1))_u=K_2\diamond_u K_1.
    \end{align*}
    Part (iii) is similar. For (iv), we have
    \[
    (ab)\diamond_u K_1\approx ((ab)f,-(ab)g)_u=(a(bf),-a(bg))_u\approx a\diamond_u(b\diamond_u K_1).
    \]
    Similarly, part (v) follows from Definition \ref{graph-sum-defn}:
    \begin{align*}
        a\diamond_u K_1\diamondplus a\diamond_u K_2&\approx (af_1+af_2,-(ag_1+ag_2))_u \\
        &=(a(f_1+f_2),-a(g_1+g_2))_u\approx a\diamond_u(K_1\diamondplus_u K_2).
    \end{align*}
    The other distributive property in (v) is shown in the same way. Finally, for (vi) note that for $i=1,2$, the condition $K_i\subset L_i$ implies $f_i \leq h_i$ and $g_i\leq l_i$, where $h_i$ and $l_i$ are the overgraph and undergraph of $L_i$, respectively. The result now follows from Definition \ref{graph-sum-defn}. 
\end{proof}

\begin{remark}
    Observe that  if $a,b\geq 0$ and $K_1,K_2\in\cK^n$ satisfy $K_1|u^\perp=K_2|u^\perp$, then
    \begin{align*}
\vol_n(a\diamond_u K_1\diamondplus_u b\diamond_u K_2) &= \int_{(a\diamond_u K_1\diamondplus_u b\diamond_u K_2)|u^\perp}\ell_{a\diamond_u K_1\diamondplus_u b\diamond_u K_2}(u;x')\,dx'\\
&=\int_{K_1|u^\perp}\left[a(f_1(x')+g_1(x'))+b(f_2(x')+g_2(x'))\right]dx'\\
&=a\int_{K_1|u^\perp}\ell_{K_1}(u;x')\,dx'+b\int_{K_2|u^\perp}\ell_{K_2}(u;x')\,dx'\\
&=a\vol_n(K_1)+b\vol_n(K_2).
    \end{align*}
\end{remark}

The next result shows that the  graph sum is a binary operation on $\cK^n$.

 \begin{theorem}\label{p-sum-convex-2}
     Let $u\in\Sp$ and $a,b\geq 0$ be given, with $a,b$ not both zero. Suppose that  $K_1,K_2\in\mathcal{K}^n$ satisfy $K_1|u^\perp=K_2|u^\perp$. Then   $a\diamond_{u} K_1\diamondplus_{u} b\diamond_{u} K_2\in\cK^n$. If, furthermore, $K_1,K_2\in\cK_o^n$, then $a\diamond_u K_1\diamondplus_u b\diamond_u K_2\in\cK_o^n$.
 \end{theorem}

 \begin{proof}
Since  $K_1,K_2\in\cK^n$ and $K_1|u^\perp=K_2|u^\perp$, the functions $f_1,f_2,g_1,g_2$ are concave on $K_i|u^\perp$. Thus, since $a,b\geq 0$, the overgraph function $af_1+bf_2$ and the undergraph function $ag_1+bg_2$ are concave on $K_i|u^\perp$. Since $(a\diamond_u K_1\diamondplus_u b\diamond_u K_2)|u^\perp=K_i|u^\perp$, by definition of the graph combination this implies that $a\diamond_u K_1\diamondplus_u b\diamond_u K_2$ is convex. The fact that $a\diamond_u K_1\diamondplus_u b\diamond_u K_2$ is compact and has nonempty interior follows from the assumptions that $K_1,K_2\in\cK^n$ and $a,b\geq 0$ are not both zero, together with the Heine--Borel theorem, the definition of the graph combination, and the fact that $K_1$ and $K_2$ have nonempty interiors. If $K_1$ and $K_2$ contain the origin in their interiors, then $o\in\relint(K_1|u^\perp)\cap\relint(K_2|u^\perp)=\relint((a\diamond_u K_1\diamondplus_u b\diamond_u K_2)|u^\perp)$, and hence $o\in\intt(a\diamond_u K_1\diamondplus_u b\diamond_u K_2)$. Thus,  $a\diamond_u K_1\diamondplus_u b\diamond_u K_2\in\cK_o^n$.
    \end{proof}

    Additionally, the graph sum of two convex polytopes is a convex polytope. More precisely, we have the following

\begin{proposition}
    Let $P_1,P_2\in\cK^n$ be polytopes which satisfy $P_1|u^\perp=P_2|u^\perp$. Then for any $a,b\geq 0$, 
    the graph combination $a\diamond_u P_1\diamondplus_u b\diamond_u P_2$ is a polytope in $\R^n$. If $a,b$ are not both zero, then $a\diamond_u P_1\diamondplus_u b\diamond_u P_2\in\cK^n$.
\end{proposition}

\begin{proof}
Let $P_1\approx(f_1,-g_1)$ and $P_2\approx(f_2,-g_2)$. Since $P_1,P_2$ are convex polytopes, the functions $f_1,f_2,g_1,g_2$ are piecewise affine and concave on their common domain $P_1|u^\perp=P_2|u^\perp$. Thus, for any $a,b\geq 0$ which are not both zero, the functions $af_1+bf_2$ and $ag_1+bg_2$ are piecewise affine and concave with the same domain. Since $a\diamond_u P_1\diamondplus_u b\diamond_u P_2\approx (af_1+bf_2,-(ag_1+bg_2))_u$ and $(a\diamond_u P_1\diamondplus_u b\diamond_u P_2)|u^\perp=P_1|u^\perp=P_2|u^\perp$, this implies that  $a\diamond_u P_1\diamondplus_u b\diamond_u P_2$ is a polytope in $\R^n$. If $a,b$ are not both zero, then $a\diamond_u P_1\diamondplus_u b\diamond_u P_2$ has nonempty interior, and hence $a\diamond_u P_1\diamondplus_u b\diamond_u P_2\in\cK^n$.
\end{proof}


\subsection{General affine surface areas of graph combinations}\label{section4}
	
	In this subsection, we show that the general $L_\phi$ (respectively, $L_\psi$) affine surface areas are concave (respectively, convex) with respect to the graph sum. These inequalities  generalize fundamental results due to Ye  \cite{Ye2013} on the monotone  behavior of the general affine surface areas under Steiner symmetrizations.  To formulate and prove our results, we will need a few more  ingredients. 
 
 For a (smooth enough) function $f:
	\bbR^{n-1}\rightarrow \bbR$, we set $\langle f(x)\rangle:=f(x)-\langle
	x, \nabla f(x)\rangle$, where  $\nabla f$ is the gradient of $f$. Note that for any two (smooth enough) functions $f, g:
	\bbR^{n-1}\rightarrow \bbR$ and all $a,b\in\R$,  we have
	$\langle af(x)+bg(x)\rangle=a \langle f(x)\rangle+b\langle
	g(x)\rangle$. The Hessian matrix of $f(x)$ is denoted by  $\hess(f(x))$, and  we use $\det (T)$ to denote 
	the determinant of a linear transformation $T$. 

 Let  $\fconc$ denote the class of functions $\phi: (0,\infty)\rightarrow
	(0,\infty)$ such that either
	$\phi$ is a nonzero constant function or $\phi$ is concave with $\lim _{t\rightarrow 0^+} \phi(t)=0$ and  $\lim
	_{t\rightarrow \infty}\frac{\phi(t)}{t}=0$, and set $\phi(0)=0$. For $K\in\cK_o^n$, the \emph{$L_\phi$ affine surface area} $\as_\phi(K)$ of $K$ is defined by	
	\begin{equation}\label{as-phi-def}
		\as_\phi(K)=\int _{\partial K} \phi
		\left(\frac{\k_K (x)}{\langle x, N_K(x)\rangle^{n+1}}\right)
		\langle x, N_K(x)\rangle\,dS_K(x),
	\end{equation}
	provided the  integral exists. Here $\partial K$ denotes the boundary of $K$, $\kappa_K(x)$ is the Gaussian curvature of $K$ at $x\in\partial K$, $N_K(x)$ is the outer unit normal of $K$ at $x$ and $S_K$ is the classical surface area measure of $K$. 
 
	The  $L_\phi$ affine surface area originated in valuation theory, where it was used to prove a fundamental result on the classification of upper semicontinuous ${\rm SL}(n)$-invariant valuations on $\mathcal{K}_o^n$.  More specifically, Ludwig and Reitzner \cite{LudR} proved that a functional $\Phi:\mathcal{K}_o^n\to\R$ is an upper semicontinuous, ${\rm SL}(n)$-invariant valuation which vanishes on polytopes if and only if there exists a concave function $\phi:[0,\infty)\to[0,\infty)$ with $\lim_{t\to 0^+}\phi(t)=\lim_{t\to\infty}\phi(t)/t=0$ such that $\Phi(K)=\as_\phi(K)$.  Here the upper semicontinuity of the $L_\phi$ affine surface area means that for every sequence of convex bodies $\{K_i\}_{i=1}^\infty\subset\cK_o^n$ which converges to $K$ in the Hausdorff metric, we have
 \[
 \limsup_{i\rightarrow \infty} \as_{\phi} (K_i)\leq \as_{\phi} (K).
 \]

	The next result is due to Ye \cite{Ye2013}. It connects the $L_\phi$ affine surface area of a convex body to its overgraph and undergraph.
	
	\bl\label{Equivalent:definition:affine:surface:area}\cite[Lemma 2.1]{Ye2013}  Let
	$K\in\cK_o^n$.  For all $\phi\in \fconc$, it holds that 
	$$\as_{\phi}(K)=\int _{K|e_n^{\perp}}\left\{\phi\left(\frac{|\det(\hess(f_{e_n}(x')))|}{\langle f_{e_n}(x')\rangle^{n+1}}\right) \langle f_{e_n}(x')\rangle
	+\phi\left(\frac{|\det( \hess(g_{e_n}(x')))|}{\langle
		g_{e_n}(x')\rangle^{n+1}}\right) \langle g_{e_n}(x')\rangle \right\}\,dx'.$$
	\el

      Let $\fconv$ denote the class of functions $\psi: (0,\infty)\rightarrow (0,\infty)$ such that either $\psi$ is a nonzero constant function, or $\psi$ is convex with $\lim_{t\rightarrow 0^+} \psi(t)=\infty$ and $ \lim_{t\rightarrow \infty} \psi(t)=0$ (in this case, we set $\psi(0)=\infty$). For a convex body $K\in\mathcal{K}_o^n$ and $\psi\in\fconv$, the \emph{$L_\psi$ affine surface area} $\as_\psi(K)$ was defined by Ludwig  \cite{Ludwig2009} to be
	\[
	\as_\psi(K)=\int _{\partial K} \psi
	\left(\frac{\k_K (x)}{\langle x, N_K(x)\rangle^{n+1}}\right)
	\langle x, N_K(x)\rangle\,dS_K(x),
	\]
provided the integral exists. Ludwig \cite[Theorem 6]{Ludwig2009} proved that if  $\psi\in\fconv$, then the $L_{\psi}$ affine surface area  is  lower  semicontinuous, that is,
			for every sequence of convex bodies $\{K_i\}_{i=1}^\infty\subset\cK_o^n$ converging to $K$ in the
			Hausdorff metric, 
   \[
   \liminf_{i\rightarrow \infty} \as_{\psi} (K_i)\geq \as_{\psi} (K).
   \]
   
The following result is also due to Ye \cite{Ye2013}.  
\begin{lemma}\label{as-psi-deping-lemma}\cite[Lemma 2.2]{Ye2013}
    For every 	$K\in\cK_o^n$ which has positive Gaussian curvature almost everywhere and every $\psi\in\fconv$, it holds that
	$$\as_{\psi}(K)=\int _{K|e_n^{\perp}}\left\{\psi\left(\frac{|\det(\hess(f_{e_n}(x')))|}{\langle f_{e_n}(x')\rangle^{n+1}}\right) \langle f_{e_n}(x')\rangle
	+\psi\left(\frac{|\det(\hess(g_{e_n}(x')))|}{\langle
		g_{e_n}(x')\rangle^{n+1}}\right) \langle g_{e_n}(x')\rangle \right\}\,dx'.$$
\end{lemma}
  
  	For more background on the general affine surface areas and their numerous applications, we refer the reader to, for example, \cite{A1, A2,CaglarYe2016, GaZ, HaberlFranz2009, Hoehner2022, Hug1995, Lei1986,LudR,Lu1, MW2,SW5,WG2007,Werner2012a,werner2008new,WernerYe2010,Ye2012,Ye2015b,Ye2016a,Zhang2007}. 

   \vspace{2mm}

	The next lemma is  a general form of Minkowski's determinant inequality (see, for example, \cite[Lemma 8.C.1]{BFR-book}).
	
	\bl\label{det-lemma} Let $A$ and $B$ be two $k\times k$ symmetric, positive semidefinite matrices. Then for any $a,b\geq 0$ which are not both zero and all $m\geq k$,
 \[
\det(aA+bB)^{\frac{1}{m}} \geq (a+b)^{\frac{k}{m}-1}\left(a\det(A)^{\frac{1}{m}}+b\det(B)^{\frac{1}{m}}\right).
 \]
If $A$ and $B$ are positive definite, then equality holds if and only if $A=B$. 
 \el

 \begin{proof}
  We adapt the proof of Lemma 8.C.1 in \cite{BFR-book}. We may assume that $m=k$, and that $A=(a_{ij})$ and $B=(b_{ij})$ are positive definite, and hence are diagonal matrices with positive eigenvalues. We first show that for any $\lambda\in[0,1]$, 
  \begin{align}\label{matrix-ineq}
     \lambda\det(A)^{\frac{1}{m}}+(1-\lambda)\det(B)^{\frac{1}{m}} &=\lambda\left(\prod_{i=1}^m a_{ii}\right)^{\frac{1}{m}}+(1-\lambda)\left(\prod_{i=1}^m b_{ii}\right)^{\frac{1}{m}} \nonumber \\
     &\leq \left(\prod_{i=1}^m(\lambda a_{ii}+(1-\lambda) b_{ii})\right)^{\frac{1}{m}}=\det(\lambda A+(1-\lambda)B)^{\frac{1}{m}}.
  \end{align}
  By the AM-GM inequality,
  \begin{align*}
      &\lambda\left(\prod_{i=1}^m\frac{a_{ii}}{\lambda a_{ii}+(1-\lambda)b_{ii}}\right)^{\frac{1}{m}}+(1-\lambda)\left(\prod_{i=1}^m\frac{b_{ii}}{\lambda a_{ii}+(1-\lambda)b_{ii}}\right)^{\frac{1}{m}}\\
      &\leq\frac{\lambda}{m}\sum_{i=1}^m\frac{a_{ii}}{\lambda a_{ii}+(1-\lambda)b_{ii}}+\frac{1-\lambda}{m}\sum_{i=1}^m\frac{b_{ii}}{\lambda a_{ii}+(1-\lambda)b_{ii}}=1.
  \end{align*}
This proves \eqref{matrix-ineq}. 

For the general case, let $\lambda=a/(a+b)$. Since $a,b\geq 0$ are not both zero, we have $\lambda\in[0,1]$. Thus, by \eqref{matrix-ineq}  we get
\begin{align*}
    \det(aA+bB)^{\frac{1}{m}}
    &=\det\left((a+b)\left(\frac{a}{a+b}A+\frac{b}{a+b}B\right)\right)^{\frac{1}{m}}\\
    &=(a+b)^{\frac{k}{m}}\det\left(\frac{a}{a+b}A+\frac{b}{a+b}B\right)^{\frac{1}{m}}\\
    &\geq(a+b)^{\frac{k}{m}}\left(\frac{a}{a+b}\det(A)^{\frac{1}{m}}+\frac{b}{a+b}\det(B)^{\frac{1}{m}}\right)\\
    &=(a+b)^{\frac{k}{m}-1}\left(a\det(A)^{\frac{1}{m}}+b\det(B)^{\frac{1}{m}}\right).
\end{align*}
 \end{proof}

 \subsubsection{Results for the $L_\phi$ general affine surface areas}
	
	The main result of this section is the following concavity property of the $L_\phi$ surface area with respect to the graph sum.

 	\bt\label{BMaffine}
 Let $\phi\in\fconc$ and $a,b\geq 0$, with $a,b$ not both zero. 
   Suppose that the function $G(t)=\phi(t^{n+1})$ is concave for $t\in(0,\infty)$. Let $K_{1},K_{2}\in\mathcal{K}_o^n$  be such that $K_1|u^\perp=K_2|u^\perp$. Then we have
			\[
   \as_\phi\left(a\diamond_u K_1\diamondplus_u b\diamond_u K_2\right) \geq \frac{a}{a+b}\cdot\as_\phi((a+b)\diamond_u K_1)+\frac{b}{a+b}\cdot\as_\phi((a+b)\diamond_u K_2).\]
	\et

 \begin{proof} 
The argument follows along the same lines as Ye's original proof of \cite[Theorem 3.5]{Ye2013}. To begin, recall that $\as_\phi$ is affine invariant, i.e., for all linear transformations $T$ with $|\det(T)|=1$, we have $\as_\phi(TK)=\as_\phi(K)$. Thus, without loss of generality, we may assume that $u=e_n$.

For $i=1,2$, let $K_i\approx(f_i,-g_i)_{e_n}$. For $x'\in K_1|e_n^\perp=K_2|e_n^\perp$, denote by
		\begin{align*}
		\widetilde{f}_{e_n}(x')&:=
		af_{1}(x')+bf_{2}(x')\\
		\widetilde{g}_{e_n}(x')&:=ag_{1}(x')+bg_{2}(x')
		\end{align*} 
         the overgraph and undergraph functions of $a\diamond_{e_n} K_1\diamondplus_{e_n} b\diamond_{e_n} K_2$, respectively.  Note that  $\langle\widetilde{f}_{e_n}(x)\rangle=a\langle f_1(x)\rangle+b\langle f_2(x)\rangle$ and $\langle\widetilde{g}_{e_n}(x)\rangle=a\langle g_1(x)\rangle+b\langle g_2(x)\rangle$.
		By Lemma \ref{det-lemma} with $k=n-1$ and $m=n+1$, for almost all $x'\in K_1|e_n^\perp$,
		\begin{align}\label{det-ineq1}
		\left|\det\left(\hess(\widetilde{f}_{e_n}(x'))\right)\right|^{\frac{1}{n+1}} &= \left|\det\left(a\hess(f_1(x'))+b\hess(f_2(x'))\right)\right|^{\frac{1}{n+1}} \nonumber\\
  &\geq (a+b)^{-\frac{2}{n+1}}\left[a\left|\det\left(\hess(f_1(x'))\right)\right|^{\frac{1}{n+1}}+b\left|\det\left(\hess(f_2(x'))\right)\right|^{\frac{1}{n+1}}\right].
		\end{align}  
  By the definition of $G$, 
  \begin{equation}\label{det-identity1}
\phi\left(\frac{|\det(\hess(\widetilde{f}_{e_n}(x')))|}{\langle\widetilde{f}_{e_n}(x')\rangle^{n+1}}\right)=G\left(\frac{|\det(\hess(\widetilde{f}_{e_n}(x')))|^{\frac{1}{n+1}}}{\langle\widetilde{f}_{e_n}(x')\rangle}\right)
=G\left(\frac{|\det(\hess(\widetilde{f}_{e_n}(x')))|^{\frac{1}{n+1}}}{a\langle f_1(x')\rangle+b\langle f_2(x')\rangle}\right).
  \end{equation}
		Since $\phi$ is an increasing function on $(0,\infty)$, so too is $G$. Thus, by \eqref{det-identity1}, \eqref{det-ineq1} and the concavity of $G(t)$ on $(0,\infty)$, we get 
		\begin{align*} \nonumber
				\phi\left(\frac{|\det(\hess(\widetilde{f}_{e_n}(x')))|}{\langle \widetilde{f}_{e_n}(x')\rangle^{n+1}}\right)
   &\geq G\left(\frac{(a+b)^{-\frac{2}{n+1}}\left[a\left|\det(\hess(f_1(x')))\right|^{\frac{1}{n+1}}+b\left|\det(\hess(f_2(x')))\right|^{\frac{1}{n+1}}\right]}{a\langle f_1(x')\rangle+b\langle f_2(x')\rangle}\right)\\
		&\geq \frac{a\langle f_1(x')\rangle}{a\langle f_1(x')\rangle+b\langle f_2(x')\rangle}\cdot G\left(\frac{(a+b)^{-\frac{2}{n+1}}|\det(\hess(f_1(x')))|^{\frac{1}{n+1}}}{\langle f_1(x')\rangle}\right)\\
  &+\frac{b\langle f_2(x')\rangle}{a\langle f_1(x')\rangle+b\langle f_2(x')\rangle}\cdot G\left(\frac{(a+b)^{-\frac{2}{n+1}}|\det(\hess(f_2(x')))|^{\frac{1}{n+1}}}{\langle f_2(x')\rangle}\right)\\
  &= \frac{a\langle f_1(x')\rangle}{a\langle f_1(x')\rangle+b\langle f_2(x')\rangle}\cdot G\left(\frac{|\det(\hess((a+b)f_1(x')))|^{\frac{1}{n+1}}}{\langle(a+b) f_1(x')\rangle}\right)\\
  &+\frac{b\langle f_2(x')\rangle}{a\langle f_1(x')\rangle+b\langle f_2(x')\rangle}\cdot G\left(\frac{|\det(\hess((a+b)f_2(x')))|^{\frac{1}{n+1}}}{\langle (a+b)f_2(x')\rangle}\right).
		\end{align*}
        
        Similarly,
		\begin{align*} \nonumber
\phi\left(\frac{|\det(\hess(\widetilde{g}_{e_n}(x')))|}{\langle \widetilde{g}_{e_n}(x')\rangle^{n+1}}\right)
			&\geq \frac{a\langle g_1(x')\rangle}{a\langle g_1(x')\rangle+b\langle g_2(x')\rangle}\cdot G\left(\frac{|\det(\hess((a+b)g_1(x')))|^{\frac{1}{n+1}}}{\langle (a+b)g_1(x')\rangle}\right)\\
  &+\frac{b\langle g_2(x')\rangle}{a\langle g_1(x')\rangle+b\langle g_2(x')\rangle}\cdot G\left(\frac{|\det(\hess((a+b)g_2(x')))|^{\frac{1}{n+1}}}{\langle (a+b)g_2(x')\rangle}\right).
		\end{align*}
		Therefore, using also Lemma \ref{Equivalent:definition:affine:surface:area} and Proposition \ref{proposition}(i),  and using $K_1|e_n^\perp=K_2|e_n^\perp$, we obtain 
		\begin{align*} 
			&\phantom{=}\as_{\phi}\left(a\diamond_{e_n} K_1\diamondplus_{e_n} b\diamond_{e_n} K_2\right)\\
			&=\int_{(a\diamond_{e_n} K_1\diamondplus_{e_n} b\diamond_{e_n} K_2)|e_n^\perp}\left[\phi\left(\frac{|\det(\hess(\widetilde{f}_{e_n}(x')))|}{\langle \widetilde{f}_{e_n}(x')\rangle^{n+1}}\right)\langle \widetilde{f}_{e_n} (x')\rangle
			+\phi\left(\frac{|\det(\hess(\widetilde{g}_{e_n}(x')))|}{\langle  \widetilde{g}_{e_n}(x')\rangle^{n+1}}\right)\langle  \widetilde{g}_{e_n} (x')\rangle\right] dx'\\
            &=\int_{(K_1|e_n^\perp)\cap(K_2|e_n^\perp)}\phi\left(\frac{|\det(\hess(\widetilde{f}_{e_n}(x')))|}{\langle \widetilde{f}_{e_n}(x')\rangle^{n+1}}\right)\langle \widetilde{f}_{e_n} (x')\rangle\, dx'\\
			&\hspace{4cm}+\int_{(K_1|e_n^\perp)\cap(K_2|e_n^\perp)}\phi\left(\frac{|\det(\hess( \widetilde{g}_{e_n}(x')))|}{\langle  \widetilde{g}_{e_n}(x')\rangle^{n+1}}\right)\langle  \widetilde{g}_{e_n} (x')\rangle\, dx'\\
			&\geq a\int_{K_1|e_n^\perp}\bigg[\langle f_1(x')\rangle\cdot G\left(\frac{|\det(\hess((a+b)f_1(x')))|^{\frac{1}{n+1}}}{\langle (a+b)f_1(x')\rangle}\right)\\
   &\hspace{4cm}+\langle g_1(x')\rangle\cdot G\left(\frac{|\det(\hess((a+b)g_1(x')))|^{\frac{1}{n+1}}}{\langle (a+b)g_1(x')\rangle}\right)\bigg]\,dx'\\
   &+b\int_{K_2|e_n^\perp}\bigg[\langle f_2(x')\rangle\cdot G\left(\frac{|\det(\hess((a+b)f_2(x')))|^{\frac{1}{n+1}}}{\langle(a+b) f_2(x')\rangle}\right)\\
   &\hspace{4cm}+\langle g_2(x')\rangle\cdot G\left(\frac{|\det(\hess((a+b)g_2(x')))|^{\frac{1}{n+1}}}{\langle (a+b)g_2(x')\rangle}\right)\bigg]\,dx'\\
   &=\frac{a}{a+b}\cdot\as_\phi((a+b)\diamond_{e_n} K_1)+\frac{b}{a+b}\cdot\as_\phi((a+b)\diamond_{e_n} K_2).
		\end{align*}
    		\end{proof}

In the special case $a+b=1$, we obtain the following

  \bc\label{as-cor-1}
Let $\phi\in{\rm Conc}(0,\infty)$ and $\lambda\in[0,1]$ be given. Suppose that the function $G(t)=\phi(t^{n+1})$ is concave for $t\in(0,\infty)$. Let $K_1,K_2\in\cK_o^n$ be such that $K_1|e_n^\perp=K_2|e_n^\perp$. Then we have
\[
\as_\phi(\lambda\diamond_{e_n} K_1\diamondplus_{e_n}(1-\lambda)\diamond_{e_n} K_2)\geq \lambda\as_\phi(K_1)+(1-\lambda)\as_\phi(K_2).
\]
  \ec

	\begin{remark}
	Let $K\in\mathcal{K}_o^n$ and  $\phi\in\fconc$, and suppose that the function $G(t)=\phi(t^{n+1})$  is concave for $t\in
		(0, \infty)$. Ye \cite[Theorem 3.5]{Ye2013} proved that the $L_\phi$ affine surface area is monotone with respect to Steiner symmetrization: 
		\begin{equation*}\as_{\phi} (K)\leq \as_{\phi}(S_{e_n}(K)).
		\end{equation*}  
  	Recalling that $\frac{1}{2} \diamond_{e_n}  K\diamondplus_{e_n}  \frac{1}{2}\diamond_{e_n}  R_{e_n}(K)=S_{e_n}(K)$, we see that Theorem \ref{BMaffine} extends \cite[Theorem 3.5]{Ye2013} from the Steiner symmetral of a single convex body to the graph sum of two convex bodies. 
	\end{remark}
  
		A convex body $K\in\cK_o^n$ is said to have {\it curvature function} $f_K: \Sp\rightarrow \R$ if 
		\[ 
        V(L, \underbrace{K, \ldots, K}_{n-1})=\frac{1}{n}\int_{\Sp}h_L(u)f_K(u)\,d\sigma (u),\]  
        where 
		$ V(L, {K, \ldots, K})$ is the mixed volume and $\sigma$ is the classical spherical measure on $\Sp$. 	If  $K\in\cK_o^n$ has curvature function $f_K$  and  $\phi\in \fconc$,  then the
		\emph{$L_{\phi}^*$ affine surface area} $\as_{\phi}^*(K)$ (see  \cite{Ludwig2009}) has the form
		\[
  \as_{\phi}^*(K)=\int_{\Sp}\phi(f_{-n}(K,u))h_K(u)^{-n}\,d\sigma(u),
  \]
		where $f_{-n}(K,u)=h_K(u)^{n+1}f_K(u)$ is the $L_p$ curvature
		function of $K$  for $p=-n$ (see \cite{Lu1}). Ludwig  \cite{Ludwig2009} proved that if $K\in\cK_o^n$ has a curvature function, then 
	\begin{equation}\label{duality:concave}
			\as_{\phi}^*(K)=\as_{\phi}(K^\circ),
   \end{equation} 
   where $K^\circ=\{x\in\R^n:\,\forall y\in K,\,\langle x,y\rangle\leq 1\}$ is the polar body of $K$. 
   For $\psi\in\fconv$, the corresponding formulation and result holds for the $L_\psi^*$ affine surface area $\as_\psi^*(K)$ as well \cite{Ludwig2009}. We thus obtain the following 

		\bc\label{uniqueness-1-1-1}
		Let $\phi\in \fconc$  be given. Let $K_{1},K_{2}\in\mathcal{K}^n_o$ be convex bodies with curvature functions  and which satisfy $K_1|e_n^\perp=K_2|e_n^\perp$, and suppose that $[\lambda\diamond_{e_n} K_1^\circ\diamondplus_{e_n}(1-\lambda)\diamond_{e_n} K_2^\circ]^\circ$ has a curvature function.  If the function $G(t)=\phi(t^{n+1})$ is concave  for $t\in(0,\infty)$, then for every $\lambda\in[0,1]$ we have
			\[
			\lambda\as_\phi^*(K_1)+(1-\lambda)\as_\phi^*(K_2)\leq \as_\phi^*([\lambda\diamond_{e_n} K_1^\circ\diamondplus_{e_n}(1-\lambda)\diamond_{e_n} K_2^\circ]^\circ).
			\]
   \ec
		
		\begin{remark}
		Let $\phi\in\fconc$ and suppose that
			$G(t)=\phi(t^{n+1})$ is concave for $t\in (0, \infty)$. Ye \cite[Corollary 3.6]{Ye2013} proved that  for every $K\in\mathcal{K}_o^n$ with a curvature function and for which $[S_{e_n}(K^\circ)]^\circ$ has a curvature function, we have 
			\[
			\as_{\phi}^{*}(K) \leq \as_{\phi}^{*}\left([S_{e_n}(K^{\circ})]^{\circ}\right).
			\]Thus, in the special case  $K_2=R_{e_n}(K_1)$ and $\lambda=1/2$,  Corollary \ref{uniqueness-1-1-1} reduces to  \cite[Corollary 3.6]{Ye2013}.
   \end{remark}

   
\subsubsection{Results for the $L_\psi$ general affine surface areas}

   In the case of the $L_\psi$ general affine surface areas, we deduce the following results. Their proofs follow in completely analogous fashion to those above for the $L_\phi$ general surface areas, and hence are omitted.

   	\bt\label{BMaffine-2}
 Let $\psi\in\fconv$ and $a,b\geq 0$, with $a,b$ not both zero. 
   Suppose that the function $G(t)=\psi(t^{n+1})$ is convex for $t\in(0,\infty)$. Let $K_{1},K_{2}\in\mathcal{K}_o^n$ be such that $K_1|e_n^\perp=K_2|e_n^\perp$, and assume that $K_1,K_2$ have positive Gaussian curvature almost everywhere. Then 
			\[
 \as_\psi\left(a\diamond_{e_n} K_1\diamondplus_{e_n} b\diamond_{e_n} K_2\right) \leq \frac{a}{a+b}\cdot\as_\psi((a+b)\diamond_{e_n} K_1)+\frac{b}{a+b}\cdot\as_\psi((a+b)\diamond_{e_n} K_2).
			\]
	\et

In the special case $a+b=1$, we obtain the following
      \bc\label{as-phi-cor-1}
Let $\psi\in\fconv$ and $\lambda\in[0,1]$ be given. Suppose that the function $G(t)=\psi(t^{n+1})$ is convex for $t\in(0,\infty)$. Let $K_1,K_2\in\cK_o^n$ be such that $K_1|e_n^\perp=K_2|e_n^\perp$, and assume that $K_1,K_2$ have positive Gaussian curvature almost everywhere. Then we have
\[
\as_\psi(\lambda\diamond_{e_n} K_1\diamondplus_{e_n}(1-\lambda)\diamond_{e_n} K_2)\leq \lambda\as_\psi(K_1)+(1-\lambda)\as_\psi(K_2).
\]
  \ec

 	\begin{remark}
	Let $K\in\mathcal{K}_o^n$ and  $\psi\in\fconv$, and suppose that the function $G(t)=\phi(t^{n+1})$  is convex for $t\in
		(0, \infty)$. Ye \cite[Theorem 3.10]{Ye2013} proved that if $K$ has positive Gaussian curvature almost everywhere, then
		\begin{equation*}\as_{\psi} (K)\geq \as_{\psi}(S_{e_n}(K)).
		\end{equation*}  
  	Again recalling that $\frac{1}{2} \diamond_{e_n}  K\diamondplus_{e_n}  \frac{1}{2}\diamond_{e_n}  R_{e_n}(K_1)=S_{e_n}(K)$, we see that Corollary \ref{as-phi-cor-1} extends \cite[Theorem 3.10]{Ye2013} from the case of the Steiner symmetral of a single convex body to the graph sum of two convex bodies. 
	\end{remark}

		\bc\label{uniqueness-1-1-2}
		Let $\psi\in \fconv$  be given. Let $K_{1},K_{2}\in\mathcal{K}^n_o$ be convex bodies with curvature functions  and which satisfy $K_1|e_n^\perp=K_2|e_n^\perp$, and suppose that $[\lambda\diamond_{e_n} K_1^\circ\diamondplus_{e_n}(1-\lambda)\diamond_{e_n} K_2^\circ]^\circ$ has a curvature function. If the function $G(t)=\psi(t^{n+1})$ is convex  for $t\in(0,\infty)$, then 
			\[
			\lambda\as_\psi^*(K_1)+(1-\lambda)\as_\psi^*(K_2)\geq \as_\psi^*([\lambda\diamond_{e_n} K_1^\circ\diamondplus_{e_n}(1-\lambda)\diamond_{e_n} K_2^\circ]^\circ).
			\]
   \ec

		\begin{remark}
		Let $\psi\in\fconv$ and suppose that
			$G(t)=\psi(t^{n+1})$ is convex for $t\in (0, \infty)$. Ye \cite[Corollary 3.11]{Ye2013} proved that for every $K\in\mathcal{K}_o^n$ with a curvature function and for which $[S_{e_n}(K^\circ)]^\circ$ has a curvature function, we have 
			\[
			\as_{\psi}^{*}(K) \geq \as_{\psi}^{*}\left([S_{e_n}(K^{\circ})]^{\circ}\right).
			\]Thus, in the special case  $K_2=R_{e_n}(K_1)$,  Corollary \ref{uniqueness-1-1-2} reduces to   \cite[Corollary 3.11]{Ye2013}.
   \end{remark}
\subsection{Minkowski's first inequality for the $L_p$ affine surface area of a graph sum}

For $K\in\cK_o^n$ and $p\neq -n$, the \emph{$L_p$ affine surface area} $\as_p(K)$	is defined by
     \begin{equation}
         \as_p(K)=\int_{\partial K}\frac{\kappa_K(x)^{\frac{p}{n+p}}}{\langle x,N_K(x)\rangle^{\frac{n(p-1)}{n+p}}}\,d\mu_{\partial K}(x),
     \end{equation}
provided the integral exists. The $L_p$ affine surface area is a key affine invariant in the $L_p$ Brunn--Minkowski theory. It was introduced by Lutwak \cite{Lu1} for $p> 1$, and by Sch\"utt and Werner \cite{SW5} for $p<1$. In the special case $p=1$, the functional $\as_1$ is the classical affine surface area, originally due to Blaschke \cite{Bl1} for sufficiently smooth convex bodies. In addition to its fundamental importance in affine differential geometry, the classical affine surface area has found numerous applications, including to, e.g.,  valuation theory \cite{A1,A2} and the approximation of convex bodies by polytopes \cite{Boe-2000,glasgrub,ludwig1999}. It was also studied in the affine plateau problem, solved in $\R^3$ by Trudinger and Wang \cite{TW2004}.

     The $L_p$ affine surface area is  invariant under affine transformations, and it is homogeneous of degree $n(n-p)/(n+p)$, that is,
     \[
\forall t>0,\quad \as_p(tK)=t^{\frac{n(n-p)}{n+p}}\as_p(K).
     \]We remark also that the $L_p$ affine surface area itself has found numerous applications, including to, e.g., the  information theory of convex bodies \cite{JW,PW,SW5,Werner2012a} and $L_p$ affine isoperimetric inequalities \cite{lutwak2000lp,werner2008new,WernerYe2010}. 

Given $u\in\Sp$ and any $K_1,K_2\in\cK_o^n$  satisfying $K_1|u^\perp=K_2|u^\perp$, we set
		\begin{equation}\label{graph-mixed-def}
		S_{u,p}^\diamondplus(K_1,K_2) :=\lim_{\varepsilon \to 0^+} \frac{\as_p(K_1\diamondplus_{u}\varepsilon \diamond_{u} K_2) -\as_p(K_1)}{\varepsilon}.
		\end{equation}
        
     \noindent  The homogeneity of the $L_p$ affine surface area allows us to deduce the following Minkowski's first inequality for a graph combination of convex bodies.  
     	\bt 
		Let  $p\in (0,1)$. For all  $K_1,K_2\in\cK_o^n$ satisfying  $K_1|u^\perp=K_2|u^\perp$, we have   
			\begin{equation}\label{MFI-asp}
			S_{u,p}^\diamondplus(K_1, K_2) \geq \as_p(K_2)+\frac{(n-p)(n-1)}{n+p}\cdot\as_p(K_1).
			\end{equation}
			If, furthermore, $\as_p(K_1) = \as_p(K_2)>0$, then
			\[
			S_{u,p}^\diamondplus(K_1,K_2) \geq \frac{n(n-p)}{n+p}\cdot\as_p(K_1).
			\] 
            If $p\in (-n,0)$  and $K_1,K_2$ have positive curvature almost everywhere, then the previous inequalities reverse.
		\et

        \begin{proof}
            Let $p\in(0,1)$,  and for $t\in(0,\infty)$, set  $\phi_p(t):=t^{\frac{p}{n+p}}$. Note that $\phi_p\in{\rm Conc}(0,\infty)$. Since $0<(np+p)/(n+p)<1$, the function $G_p(t)=\phi_p(t^{n+1})$ is strictly concave for $t\in(0,\infty)$. Thus, the hypotheses of Theorem \ref{BMaffine} are satisfied by the function $\phi_p$.

            Next, we again adapt the standard arguments to prove the Minkowski's first inequality, modified for the setting of the graph sum.  Substituting $\varepsilon=t/(1-t)$ into \eqref{graph-mixed-def}, and using Proposition \ref{graph-proposition} and the homogeneity of the $L_p$ affine surface area, we derive that
            \begin{align*}
                S_{u,p}^\diamondplus(K_1,K_2) &=\lim_{t\to 0^+}\frac{\as_p\left(\frac{1}{1-t}\diamond_u[(1-t)\diamond_u K_1\diamondplus_u t\diamond_u K_2]\right)-\as_p(K_1)}{t/(1-t)} \\
                &=\lim_{t\to 0^+}\frac{(1-t)^{-\frac{n(n-p)}{n+p}}\as_p\left((1-t)\diamond_u K_1\diamondplus_u t\diamond_u K_2\right)-\as_p(K_1)}{t/(1-t)} \\
                 &=\lim_{t\to 0^+}\frac{\as_p\left((1-t)\diamond_u K_1\diamondplus_u t\diamond_u K_2\right)-(1-t)^{\frac{n(n-p)}{n+p}}\as_p(K_1)}{t(1-t)^{\frac{(n-p)(n-1)}{n+p}}} \\
                 &=\lim_{t\to 0^+}\frac{\as_p((1-t)\diamond_u K_1\diamondplus_u t\diamond_u K_2)-\as_p(K_1)}{t}+\lim_{t\to 0^+}\frac{\left[1-(1-t)^{\frac{n(n-p)}{n+p}}\right]\as_p(K_1)}{t}\\
                &=\lim_{t\to 0^+}\frac{\as_p((1-t)\diamond_u K_1\diamondplus_u t\diamond_u K_2)-\as_p(K_1)}{t}+\frac{n(n-p)}{n+p}\cdot\as_p(K_1).
            \end{align*}
    We now set $f_p(t):=\as_p((1-t)\diamond_u K_1\diamondplus_u t\diamond_u K_2)$. Note that $f_p(0)=\as_p(K_1)$ and $f_p(1)=\as_p(K_2)$.  The desired inequality \eqref{MFI-asp} is equivalent to $f_p'(0)\geq f_p'(1)-f_p'(0)$, and this latter inequality holds by Theorem \ref{BMaffine}. Therefore,
    \[
S_{u,p}^\diamondplus(K_1,K_2)-\frac{n(n-p)}{n+p}\cdot\as_p(K_1)\geq\as_p(K_2)-\as_p(K_1),
    \]
    which is equivalent to \eqref{MFI-asp}. The result follows for the case $p\in(0,1)$.

            In the case $p\in(-n,0)$,  for $t\in(0,\infty)$ we set $\psi_p(t):=t^{\frac{p}{n+p}}$, and note that $\psi_p\in{\rm Conv}(0,\infty)$. Since $(np+p)/(n+p)<0$, the function $G_p(t)=\psi_p(t^{n+1})$ is convex on $(0,\infty)$. Thus, the hypotheses of Theorem \ref{BMaffine-2} are satisfied, and the rest of the proof follows in the same way as before.
        \end{proof}
 

\section{$L_p$ chord combinations of convex bodies}\label{Lp-chord-section}

Recall that  for $a,b,s,t>0$ and $p\in[-\infty,+\infty]$, the \emph{$p$-means} $M_p^{(a,b)}(s,t)$ are defined by
\begin{equation}\label{pmeans-def}
M_p^{(a,b)}(s,t)=\begin{cases}
(as^p+bt^p)^{1/p}, &\text{ if }p\in\R\setminus\{0\};\\
s^a t^b, &\text{ if }p=0;\\
\min\{s,t\}, &\text{ if }p=-\infty;\\
\max\{s,t\}, &\text{ if }p=+\infty.
\end{cases}
\end{equation}

\noindent The main concept of this section is the following

\begin{defn}\label{main-def}  
	Let $K_1$ and $K_2$ be convex bodies in $\R^n$, $a,b>0$, $p\in[-\infty,+\infty]$, and  $u\in\Sp$. The \emph{$L_p$ chord combination} $a\odot_{u,p} K_1\oplus_{u,p} b\odot_{u,p} K_2$ is defined so that the chord $(a\odot_{u,p}K_1\oplus_{u,p}b\odot_{u,p}K_2)\cap(x'+\R u)$  in the direction $u$ and passing through $x'\in u^\perp$  has length
 \begin{align*}
   \ell_{a\odot_{u,p}K_1\oplus_{u,p}b\odot_{u,p}K_2}(u;x')=M_p^{(a,b)}(\ell_{K_1}(u;x'),\ell_{K_2}(u;x'))
 \end{align*}
 and is bisected by $u^\perp$. The \emph{$L_p$ chord dilation} $a\odot_{u,p}K_1$ is defined so that the chord $(a\odot_{u,p}K_1)\cap(x'+\R u)$ has length 
 \[
\ell_{a\odot_{u,p}K_1}(u;x')=a^{1/p}\ell_{K_1}(u;x')
 \]
 and is bisected by $u^\perp$.
	\end{defn}

Put another way,
\begin{align*}
     &\phantom{:=}a\odot_{u,p} K_1\oplus_{u,p} b\odot_{u,p} K_2 \\
     &:=\bigg\{x'+tu:\, x'\in u^\perp,
     -\frac{1}{2}M_p^{(a,b)}(\ell_{K_1}(u;x'),\ell_{K_2}(u;x'))\leq t\leq \frac{1}{2}M_p^{(a,b)}(\ell_{K_1}(u;x'),\ell_{K_2}(u;x'))\bigg\}.
\end{align*} 
Thus, by construction, the $L_p$ chord combination and dilation are symmetric about $u^\perp$. 
We also denote $K_1\oplus_{u,p} K_2:=1\odot_{u,p} K_1\oplus_{u,p} 1\odot_{u,p} K_2$. In particular, choosing $p=1$, $a=\lambda\in(0,1)$, $b=1-\lambda$ and $K_1=K_2=K$, we recover the classical Steiner symmetral of a convex body $K$ with respect to $u^\perp$: 
  \begin{equation}\label{steiner-Lp-recovery}
  S_u(K)=\lambda\odot_{u,1} K\oplus_{u,1}(1-\lambda)\odot_{u,1} K.
  \end{equation}
  One may also obtain the Steiner symmetral by replacing one or both instances of $K$ by $R_u(K)$ in \eqref{steiner-Lp-recovery}.  

Using the convention $1/(\pm\infty)=0$, we  have $a\odot_{u,\pm\infty}K=K+tu$ for some $t\in\R$.  Since $K_1|u^\perp=K_2|u^\perp$, we also have $K_1\oplus_{u,-\infty}K_2 = (K_1\cap K_2)+su$ and $K_1\oplus_{u,-\infty}K_2 = (K_1\cup K_2)+tu$ for some $s,t\in\R$.   

\begin{remark}
    Let $K_1,K_2\in\cK^n$ be convex bodies with common projections  $K_1|u^\perp=K_2|u^\perp$, and assume that $K_1$ and $K_2$ are symmetric about $u^\perp$, i.e., $K_1=R_u(K_1)$ and $K_2=R_u(K_2)$. Then for any $a,b\geq 0$, we have
    \[
a\diamond_u K_1\diamondplus_u b\diamond_u K_2 = a\odot_{u,1}K_1\oplus_{u,1}b\odot_{u,1}K_2.
    \]
    Thus, analogues of the results in Section \ref{graph-section} for graph combinations of convex bodies also hold for the corresponding $L_1$ chord combinations, under the additional assumptions that $K_1$ and $K_2$ are symmetric about $u^\perp$. 
    
\end{remark}

\subsection{Examples of $L_p$ chord sums} 

Consider again the planar convex bodies  $K_1=[-1,1]\times[-1,1]$ and $K_2=B_2^2$, and let $u=e_2$. For $x_1\in[-1,1]$, we have
\[
\ell_{K_1\oplus_{u,p}K_2}(u;x_1) =\begin{cases}
    2\left[1+(1-x_1^2)^{\frac{p}{2}}\right]^{\frac{1}{p}}, &\text{if } p\in\R\setminus\{0\};\\
    4\sqrt{1-x_1^2}, &\text{if }p=0;\\
    2\sqrt{1-x_1^2},&\text{if }p=-\infty;\\
    2, &\text{if }p=+\infty.
\end{cases}
\]
Thus, for $p\in\R\setminus\{0\}$ we have
\begin{equation}\label{Lp-chord-sum-example}
K_1\oplus_{u,p}K_2 = \left\{(x_1,x_2)\in\R^2:\,-1\leq x_1\leq 1,\,-\left[1+(1-x_1^2)\right]^{\frac{p}{2}}\leq x_2\leq \left[1+(1-x_1^2)\right]^{\frac{p}{2}}\right\}.
\end{equation}
Moreover,
\begin{align*}
    K_1\oplus_{u,0}K_2 &= \left\{(x_1,x_2)\in\R^2:\,-1\leq x_1\leq 1,\,-2\sqrt{1-x_1^2}\leq x_2\leq 2\sqrt{1-x_1^2}\right\},
\end{align*}
which is an ellipse with semiaxes 1 along the $x_1$-axis and 2 along the $x_2$-axis. For $p=\pm\infty$, we have $K_1\oplus_{u,-\infty}K_2=K_2$ and $K_1\oplus_{u,+\infty}K_2=K_1$. 

Intuitively, the $L_p$ chord sum and $L_p$ fiber sum are distinct operations on $\cK^n$. On one hand, the $L_p$ chord sum operates locally on the chord lengths of the bodies to create a body which is symmetric about $u^\perp$.  On the other hand, the $L_p$ fiber sum operates globally over all directions via the infimal convolution of the support functions of the bodies. Let us demonstrate these observations with a concrete example.

For $p=1$, we have the support functions $h_{K_1\oplus_{u,1}K_2}(x_1,x_2)=|x_2|+\sqrt{x_1^2+x_2^2}$ and $h_{K_1\capplus_1 K_2}(x_1,x_2)=|x_1|+\sqrt{x_1^2+x_2^2}$. Hence, $h_{K_1\oplus_{u,1}K_2}(1,0)=1\neq 2=h_{K_1\capplus_1 K_2}(1,0)$, so $K_1\oplus_{u,1}K_2$ and $K_1\medcapplus_1 K_2$ are distinct bodies. For $p=2$, we have $h_{K_1\oplus_{u,2}K_2}(x_1,x_2)=\sqrt{2}\sqrt{x_1^2+x_2^2}$ and $h_{K_1\capplus_2 K_2}(x_1,x_2)$ is given by the piecewise function in \eqref{example-support-fn-2}. Hence, $K_1\oplus_{u,2}K_2$ and $K_1\medcapplus_2 K_2$ are distinct bodies. Therefore, in general the $L_p$ chord sum and $L_p$ fiber sum are distinct operations on $\cK^2$.

Next, we show that the $L_1$ chord sum and the graph sum are distinct operations on $\cK^2$. Intuitively, this is clear because the $L_1$ chord sum produces a body which is symmetric about $u^\perp$, whereas the graph sum does not, in general, have this property. To illustrate, consider $K_3=\{(x_1,x_2)\in\R^2:\, -1\leq x_1\leq 1,\, x_1^2-1\leq x_2\leq 1\}$. For $u=e_2$ and $x_1\in[-1,1]$, we have $\ell_{K_1\oplus_{u,1}K_3}(u;x_1)=\ell_{K_1}(u;x_1)+\ell_{K_3}(u;x_1)=4-x_1^2$. Hence,
\[
K_1\oplus_{u,1}K_3 = \left\{(x_1,x_2)\in\R^2:\,-1\leq x_1\leq 1,\,\frac{x_1^2}{2}-2\leq x_2\leq 2-\frac{x_1^2}{2}\right\}.
\]
Recall that the graph sum of $K_1$ and $K_3$ is
\[
K_1\diamondplus_u K_3=\left\{(x_1,x_2)\in\R^2:\,-1\leq x_1\leq 1,\, x_1^2-2\leq x_2\leq 2\right\}.
\]
Hence, $K_1\oplus_{u,1}K_3$ and $K_1\diamondplus_u K_3$ are distinct bodies (see Figure 3 below). 


\begin{center}
\begin{tikzpicture}[scale=1.1]
    \draw[->] (-2.5,0) -- (2.5,0) node[right] {$x_1$};
    \draw[->] (0,-3.5) -- (0,3.5) node[above] {$x_2$};

    \node at (0,0) [below left] {$o$};

    \draw[dashed, black] (1, -3.5) -- (1, 3.5);
    \node at (1, -3.6) [below, font=\small] {$x_1=1$};
    \draw[dashed, black] (-1, -3.5) -- (-1, 3.5);
    \node at (-1, -3.6) [below, font=\small] {$x_1=-1$};

    \draw[dashed, red] (-2.5, 2) -- (2.5, 2);
    \node at (-4, 2) [right, font=\small, red] {$x_2=2$};

    \draw[dashed, red, samples=100, domain=-1.5:1.5] plot (\x, {-(2-\x*\x)});
    \node at (-3.8, -0.4) [right, font=\small, red] {$x_2=-(2-x_1^2)$};

    \draw[dashed, blue, samples=100, domain=-1.5:1.5] plot (\x, {2-\x*\x/2});
    \node at (-4, 1) [right, font=\small, blue] {$x_2=2-x_1^2/2$};
    \draw[dashed, blue, samples=100, domain=-1.5:1.5] plot (\x, {-(2-\x*\x/2)});
    \node at (-4, -1.3) [right, font=\small, blue] {$x_2=-(2-x_1^2/2)$};


    \fill[blue!20, opacity=0.7]
        plot[domain=-1:1, samples=100] (\x, {2 - \x*\x/2}) --
        plot[domain=1:-1, samples=100] (\x, {- (2 - \x*\x/2)}) -- cycle;

    \draw[blue, thick, samples=100, domain=-1:1] plot (\x, {2 - \x*\x/2});
    \draw[blue, thick, samples=100, domain=-1:1] plot (\x, {- (2 - \x*\x/2)});
    \draw[blue, thick] (-1, {2 - (-1)*(-1)/2}) -- (-1, {- (2 - (-1)*(-1)/2)});
    \draw[blue, thick] (1, {2 - (1)*(1)/2}) -- (1, {- (2 - (1)*(1)/2)});

    \node[blue, anchor=south east] at (3.4, -2) {$K_1\oplus_{u,1}K_3$};


    \fill[red!20, opacity=0.5]
        (-1, 2) -- (1, 2) -- 
        plot[domain=1:-1, samples=100] (\x, {\x*\x - 2}) -- cycle;

    \draw[red, thick, samples=100, domain=-1:1] plot (\x, {\x*\x - 2});
    \draw[red, thick] (-1, 2) -- (1, 2); 
    \draw[red, thick] (-1, {(-1)*(-1) - 2}) -- (-1, 2);
    \draw[red, thick] (1, {(1)*(1) - 2}) -- (1, 2);

    \node[red, anchor=north east] at (3.2, 2.6) {$K_1\diamondplus_u K_3$};

\end{tikzpicture}
\end{center}
\noindent{\footnotesize{\bf Figure 3:} The graph sum $K_1\diamondplus_u K_3$ and the $L_1$ chord sum $K_1\oplus_{u,1}K_3$ are illustrated in this figure. Observe that $K_1\diamondplus_u K_3$ is not symmetric about $u^\perp$.}

 \subsection{Properties of $L_p$ chord combinations}
	
	Next, we highlight some  basic properties of  $L_p$ chord combinations of convex bodies.
	
	\begin{proposition}\label{proposition}  Let $K_1,K_2\in\mathcal{K}^n$, $p\in[-\infty,+\infty]$, $u\in\Sp$ and $a,b>0$.

		\begin{itemize}
			\item [(i)] (Shadow property)  We have 
			\[
   (a\odot_{u,p} K_1\oplus_{u,p} b\odot_{u,p} K_2)|u^\perp=(K_1|u^\perp)\cap(K_2|u^\perp).
   \]

   \item[(ii)] 	(Commutativity)  It holds that  \[K_1\oplus_{u,p} K_2=K_2\oplus_{u,p} K_1.\]
			
			\item [(iii)] (Associativity of sum) Let $K_3$ also be a convex body in $\mathcal{K}^n$. Then
			$$(K_1\oplus_{u,p} K_2)\oplus_{u,p} K_3=K_1\oplus_{u,p} (K_2\oplus_{u,p}  K_3).$$ 	
			
			\item [(iv)] (Associativity of dilation) It holds that  \[(ab)\odot_{u,p} K=	a\odot_{u,p} 	(b\odot_{u,p} K).\]
             
            \item [(v)] (Distributivity) We have
            \begin{align*}
                a\odot_{u,p} (K_1\oplus_{u,p} K_2)&=a\odot_{u,p} K_1\oplus_{u,p} a\odot_{u,p} K_2
            \end{align*}
            and
            \[
            (a+b)\odot_{u,p}K_1 = a\odot_{u,p}K_1\oplus_{u,p}b\odot_{u,p}K_1.
            \]
			
			\item[(vi)] (Monotonicity) Let $L_1,L_2\in\mathcal{K}^n$ be such that $K_1\subset L_1$ and $K_2\subset L_2$. Then for every $a,b>0$, we have
			$$a\odot_{u,p} K_1\oplus_{u,p} b\odot_{u,p} K_2\subset a\odot_{u,p} L_1\oplus_{u,p} b\odot_{u,p} L_2.$$   

		\end{itemize}
	\end{proposition}
 
	\begin{proof}
		Part (i) follows immediately from Definition \ref{main-def}. Property (ii) also follows directly from Definition \ref{main-def} and the commutativity of the $p$-means \eqref{pmeans-def}: $M_p^{(a,b)}(s,t)=M_p^{(b,a)}(t,s)$. 
		To prove (iii), note that for any $x'\in u^\perp$ we have
  \begin{align*}
    \frac{1}{2}\ell_{(K_1\oplus_{u,p}K_2)\oplus_{u,p}K_3}(u;x') &=M_p^{(1,1)}\left(\ell_{K_1\oplus_{u,p}K_2}(u;x'),\ell_{K_3}(u;x')\right)\\
     &=M_p^{(1,1)}\left(\ell_{K_1}(u;x'),\ell_{K_2\oplus_{u,p}K_3}(u;x')\right)\\
      &=\frac{1}{2}\ell_{K_1\oplus_{u,p}(K_2\oplus_{u,p}K_3)}(u;x').
  \end{align*}
  Since $x'$ was arbitrary and the $L_p$ chord sum is symmetric about $u^\perp$, the claim follows.
  
  To prove (iv), note that $ab\odot_u K=a\odot_u(b\odot_u K)$ and hence
  \begin{align*}
      (ab)\odot_{u,p}K=(ab)^{\frac{1}{p}}\odot_u K=a^{\frac{1}{p}}b^{\frac{1}{p}}\odot_u K=a^{\frac{1}{p}}\odot_u\left(b^{\frac{1}{p}}\odot_u K\right)=a\odot_{u,p}(b\odot_{u,p}K).
  \end{align*}
For $p=0$, we have 
\begin{align*}
    \ell_{(ab)\circ_{u,0}K}(u;x')=\ell_K(u;x')^{ab}=(\ell_K(u;x')^b)^a=\ell_{b\circ_{u,0}K}(u;x')^a=\ell_{a\circ_{u,0}(b\circ_{u,0}K)}(u;x').
\end{align*}
Now invoke the symmetry of the $L_p$ chord combination about $u^\perp$. The cases $p=\pm\infty$ are trivial.
 
   For	(v), assume first that $p\in\R\setminus\{0\}$. Then we have
  \begin{align*}
     \frac{1}{2}\ell_{a\odot_{u,p}(K_1\oplus_{u,p}K_2)}(u;x') 
     &=\frac{1}{2}a^{\frac{1}{p}}\ell_{K_1\oplus_{u,p}K_2}(u;x')\\
     &=a^{\frac{1}{p}}M_p^{(1,1)}(\ell_{K_1}(u;x'),\ell_{K_2}(u;x'))\\
     &=M_p^{(a,a)}(\ell_{K_1}(u;x'),\ell_{K_2}(u;x'))\\
&=\frac{1}{2}\ell_{a\odot_{u,p}K_1\oplus_{u,p}a\odot_{u,p}K_2}(u;x').
  \end{align*}
Now, using this relation and  the symmetry of each $L_p$ chord combination about $u^\perp$, we obtain the desired identity. The special cases $p=0,\pm\infty$ follow along the same lines. The second identity in (v) is established in a similar way.

  For (vi), since $K_i\subset L_i$, for every $x'\in u^\perp$ we have 
		\[
		\ell_{K_i}(u;x')\leq\ell_{L_i}(u;x')\quad\text{for}\quad i=1,2.
		\]
		Hence, for all $a,b>0$ and all $p\in[-\infty,+\infty]$, we have
		\[
M_p^{(a,b)}(\ell_{K_1}(u;x'),\ell_{K_2}(u;x'))\leq M_p^{(a,b)}(\ell_{L_1}(u;x'),\ell_{L_2}(u;x')).
		\]
		The conclusion now follows from Definition \ref{main-def}.
	\end{proof}

The next result states that the $L_p$ chord combination is a binary operation on $\cK^n$ when $0< p\leq 1$. 

 \begin{theorem}\label{p-sum-convex}
     Let $p\in(0,1]$  and $u\in\Sp$ be given.  Then for all $K_1,K_2\in\mathcal{K}^n$ with $K_1|u^\perp=K_2|u^\perp$ and all $a,b>0$, we have $a\odot_{u,p} K_1\oplus_{u,p} b\odot_{u,p} K_2\in\mathcal{K}^n$. If, furthermore, $K_1,K_2\in\mathcal{K}_o^n$, then $a\odot_{u,p} K_1\oplus_{u,p} b\odot_{u,p} K_2\in\cK_o^n$.
 \end{theorem}

 We will use the following lemma.

 \begin{lemma}\label{chord-lemma}
     Let $K\in\cK^n$ and $u\in\Sp$. For any $\lambda\in[0,1]$ and any $x_1,x_2\in K|u^\perp$, we have
     \begin{equation}
         \ell_K(u;\lambda x_1+(1-\lambda)x_2) \geq \lambda\ell_K(u;x_1)+(1-\lambda)\ell_K(u;x_2).
     \end{equation}
 \end{lemma}

\begin{proof}
    We begin by showing that for any $x_1,x_2\in K|u^\perp$ and any $\lambda\in[0,1]$, we have
    \begin{equation}\label{convexity-inc}
        K\cap\left[\lambda x_1+(1-\lambda)x_2+\R u\right] \supset \lambda\left[K\cap(x_1+\R u)\right]+(1-\lambda)\left[K\cap(x_2+\R u)\right].
    \end{equation}
Let $x$ be an element of the right-hand side of \eqref{convexity-inc}. Then $x=y+z$, where $y\in\lambda[K\cap(x_1+\R u)]$ and $z\in(1-\lambda)[K\cap(x_2+\R u)]$. Hence $y=\lambda(x_1+su)$ and $z=(1-\lambda)(x_2+tu)$ for some $s,t\in\R$, where $x_1+su,x_2+tu\in K$. Therefore,
\begin{align*}
    x=y+z&=\lambda x_1+(1-\lambda)x_2+\left[\lambda s+(1-\lambda)t\right]u
    \in\lambda x_1+(1-\lambda)x_2+\R u.
\end{align*}
Since $K$ is convex, this implies that
    \[
x=y+z=\lambda(x_1+su)+(1-\lambda)(x_2+tu)\in K.
    \]
Thus, $x$ is an element of the left-hand side of \eqref{convexity-inc}. This proves \eqref{convexity-inc}.

By \eqref{convexity-inc}, the monotonicity and homogeneity of the functional $\vol_1$, and the Brunn--Minkowski inequality, we obtain
\begin{align*}
    \ell_K(u;\lambda x_1+(1-\lambda)x_2) &=\vol_1\left(K\cap[\lambda x_1+(1-\lambda)x_2+\R u]\right)\\
    &\geq\vol_1\left(\lambda[K\cap(x_1+\R u)]+(1-\lambda)[K\cap(x_2+\R u)]\right)\\
    &\geq \vol_1\left(\lambda[K\cap(x_1+\R u)]\right)+\vol_1\left((1-\lambda)[K\cap(x_2+\R u)]\right)\\
    &=\lambda\vol_1\left(K\cap(x_1+\R u)\right)+(1-\lambda)\vol_1\left(K\cap(x_2+\R u)\right)\\
    &=\lambda\ell_K(u;x_1)+(1-\lambda)\ell_K(u;x_2).
\end{align*}
\end{proof}

 \begin{proof}[Proof of Theorem \ref{p-sum-convex}]
 In view of the Heine--Borel theorem, it is clear that the compactness of $K_1$ and $K_2$ implies the compactness of $a\odot_{u,p} K_1\oplus_{u,p} b\odot_{u,p} K_2$.  Thus, it remains to prove that $a\odot_{u,p} K_1\oplus_{u,p} b\odot_{u,p} K_2$ is convex when $p\in(0,1]$.
 
    Let $(x_1',t_1),(x_2',t_2)\in a\odot_{u,p} K_1\oplus_{u,p} b\odot_{u,p} K_2$. Then:
    \begin{itemize}
        \item[(i)] $x_1',x_2'\in K_1|u^\perp=K_2|u^\perp$;

        \item[(ii)] $-\frac{1}{2}M_p^{(a,b)}\left(\ell_{K_1}(u;x_i'),\ell_{K_2}(u;x_i')\right)
        \leq t_i\leq \frac{1}{2}M_p^{(a,b)}\left(\ell_{K_1}(u;x_i'),\ell_{K_2}(u;x_i')\right)$, $i=1,2$.
    \end{itemize}
Let $\lambda\in[0,1]$ and consider the convex combination
\[
\lambda(x_1',t_1)+(1-\lambda)(x_2',t_2) = (\underbrace{\lambda x_1'+(1-\lambda)x_2'}_{=:x'_\lambda},\underbrace{\lambda t_1+(1-\lambda)t_2}_{=:t_\lambda})=(x_\lambda',t_\lambda).
\]
Then we have the following:
\begin{itemize}
    \item[(i)'] By (i) and the convexity of $(K_1|u^\perp)\cap(K_2|u^\perp)$, we have $x'_\lambda\in(K_1|u^\perp)\cap (K_2|u^\perp)=K_i|u^\perp$ for $i=1,2$.

    \item[(ii)'] Using (ii),  Minkowski's inequality for sums and Lemma \ref{chord-lemma}, we derive
    \begin{align*}
        t_\lambda &=\lambda t_1+(1-\lambda)t_2\\
        &\leq\frac{1}{2}\lambda M_p^{(a,b)}\left(\ell_{K_1}(u;x_1'),\ell_{K_2}(u;x_1')\right)
        +\frac{1}{2}(1-\lambda)M_p^{(a,b)}\left(\ell_{K_1}(u;x_2'),\ell_{K_2}(u;x_2')\right)\\
        &=\left[\left(\frac{1}{2}a^{\frac{1}{p}}\lambda\ell_{K_1}(u;x_1')\right)^p+\left(\frac{1}{2}b^{\frac{1}{p}}\lambda\ell_{K_2}(u;x_1')\right)^p\right]^{\frac{1}{p}}\\
        &+\left[\left(\frac{1}{2}a^{\frac{1}{p}}(1-\lambda)\ell_{K_1}(u;x_2')\right)^p+\left(\frac{1}{2}b^{\frac{1}{p}}(1-\lambda)\ell_{K_2}(u;x_2')\right)^p\right]^{\frac{1}{p}}
        \\
        &\leq\left[\frac{a}{2^p}\left(\lambda\ell_{K_1}(u;x_1')+(1-\lambda)\ell_{K_1}(u;x_2')\right)^p+\frac{b}{2^p}\left(\lambda\ell_{K_2}(u;x_1')+(1-\lambda)\ell_{K_2}(u;x_2')\right)^p\right]^{\frac{1}{p}}\\
        &\leq \frac{1}{2}\left[a\ell_{K_1}(u;x'_\lambda)^p+b\ell_{K_2}(u;x'_\lambda)^p\right]^{\frac{1}{p}}\\
        &=\frac{1}{2}M_p^{(a,b)}(\ell_{K_1}(u;x'_\lambda),\ell_{K_2}(u;x'_\lambda)).
    \end{align*}
    The inequality $t_\lambda\geq -\frac{1}{2}M_p^{(a,b)}(\ell_{K_1}(u;x'_\lambda),\ell_{K_2}(u;x'_\lambda))$ is handled similarly. 
\end{itemize}
By (i)' and (ii)', we have $(x'_\lambda,t_\lambda)\in a\odot_{u,p} K_1\oplus_{u,p} b\odot_{u,p} K_2$. Finally, if $K_1$ and $K_2$ contain the origin in their interiors, then $\ell_{K_i}(u;o)>0$ for $i=1,2$, so $\ell_{a\odot_{u,p} K_1\oplus_{u,p} b\odot_{u,p} K_2}(u;o)>0$, which implies that $a\odot_{u,p} K_1\oplus_{u,p} b\odot_{u,p} K_2$ contains the origin in its interior.
 \end{proof}


 \begin{remark}\label{counterexample-rmk}
     If $p>1$, then the $L_p$ chord sum $K_1\oplus_{u,p} K_2$ need not be convex. To illustrate, consider the following example in $\R^2$. Let $K_1=[-1,1]\times[-1,1]$ and $K_2=\conv\{\pm e_1,\pm e_2\}$. Then $K_1|e_2^\perp=K_2|e_2^\perp=\conv\{\pm e_1\}$, which we write as $[-1,1]$. For every $x'\in[0,1]$, we have $\ell_{K_1}(e_2;x')=2$ and $\ell_{K_2}(e_2;x')=2(1-x')$. For every $x'\in[-1,0]$, we have $\ell_{K_1}(e_2;x')=2$ and $\ell_{K_2}(e_2;x')=2(1+x')$. Thus, $K_1\oplus_{e_2,p} K_2$ has overgraph 
     \[
     f_{K_1\oplus_{e_2,p} K_2}(x')=\begin{cases} 
      \left[1+(1-x')^p\right]^{1/p}, & \text{if }x'\in[0,1]; \\
      \left[1+(1+x')^p\right]^{1/p}, & \text{if }x'\in[-1,0].
   \end{cases}
     \]
     Note that this function is not concave on $[-1,1]$. Hence, in this example, $K_1\oplus_{e_2,p} K_2$ is not convex if $p>1$. This example also shows that if $p\neq 1$, then the $L_p$ chord sum of two convex polytopes need not be a convex polytope.
 \end{remark}


\subsection{Brunn--Minkowski-type inequalities for $L_p$ chord sums of convex bodies}

The main result of this section is the following Brunn--Minkowski-type inequality for $L_p$ chord  combinations of convex bodies. 

\begin{theorem}\label{BMI-Lp-fiber-new}
    Let $K_1,K_2\in\cK^n$, $u\in\Sp$ and $a,b>0$. For $i=1,2$ and $x'\in(K_1|u^\perp)\cap(K_2|u^\perp)$, set  $f_{u,i}(x'):=\sup\{t\in\R:\, x'+tu\in K_i\}$, $g_{u,i}(x'):=-\inf\{t\in\R:\,x'+tu\in K_i\}$ and   \[
    \widetilde{K}_i^u:=\left\{(x',t):\, x'\in(K_1|u^\perp)\cap(K_2|u^\perp),\,-g_{u,i}(x')\leq t\leq f_{u,i}(x')\right\}.
    \] 
    \begin{itemize}
        \item[(i)] If $p\in(0,1)$, then 
    \begin{equation*}
        \vol_n(a\odot_{u,p} K_1\oplus_{u,p}b\odot_{u,p} K_2)^p \leq a\vol_n(\widetilde{K}_1^u)^p+b\vol_n(\widetilde{K}_2^u)^p
            \end{equation*}
    with equality  if and only if $\widetilde{K}_1^u=\frac{cb}{a}\odot_{u,p}\widetilde{K}_2^u+tu$ for some constant $c>0$ and $t\in\R$.

        \item[(ii)] If $p>1$, then 
    \begin{equation*}
        \vol_n(a\odot_{u,p} K_1\oplus_{u,p}b\odot_{u,p} K_2)^p \geq a\vol_n(\widetilde{K}_1^u)^p+b\vol_n(\widetilde{K}_2^u)^p
            \end{equation*}
with  equality if and only if $\widetilde{K}_1^u=\frac{cb}{a}\odot_{u,p}\widetilde{K}_2^u+tu$ for some constant $c>0$ and $t\in\R$.
 
        \item[(iii)]     If $p=1$, then 
      \begin{equation*}
        \vol_n(a\odot_u K_1\oplus_u b\odot_u K_2)= a\vol_n(\widetilde{K}_1^u)+b\vol_n(\widetilde{K}_2^u).
    \end{equation*}
    \end{itemize}
\end{theorem}

\begin{proof}
(i)   Let $p\in(0,1)$. For $x'\in(K_1|u^\perp)\cap(K_2|u^\perp)$, set $\widetilde{f}_u(x'):=a\ell_{K_1}(u;x')^p$ and $\widetilde{g}_u(x'):=b\ell_{K_2}(u;x')^p$. Integrating along the chords orthogonal to $u^\perp$, we obtain
\begin{align*}
    \vol(a\odot_{u,p}K_1\oplus_{u,p}b\odot_{u,p}K_2) 
&=\int_{(a\odot_{u,p}K_1\oplus_{u,p}b\odot_{u,p}K_2)|u^\perp}\ell_{a\odot_{u,p}K_1\oplus_{u,p}b\odot_{u,p}K_2}(u;x')\,dx'\\
&=\int_{(K_1|u^\perp)\cap(K_2|u^\perp)}\left(a\ell_{K_1}(u;x')^p+b\ell_{K_2}(u;x')^p\right)^{\frac{1}{p}}\,dx'\\
&=\int_{(K_1|u^\perp)\cap(K_2|u^\perp)}\left(\widetilde{f}_u(x')+\widetilde{g}_u(x')\right)^{\frac{1}{p}}\,dx'\\
&=\|\widetilde{f}_u+\widetilde{g}_u\|_{L_{\frac{1}{p}}}^{\frac{1}{p}}
\end{align*}
where $\|f\|_{L_{\frac{1}{p}}}=\left(\int_{\R^n}|f|^{\frac{1}{p}}\,dx\right)^p$ denotes the $L_{\frac{1}{p}}$ norm of a function $f:\R^n\to\R$. Since $\frac{1}{p}\geq 1$, by Minkowski's integral inequality  we have
\begin{align*}
\|\widetilde{f}_u+\widetilde{g}_u\|_{L_{\frac{1}{p}}}^{\frac{1}{p}}& \leq\left(\|\widetilde{f}_u\|_{L_{\frac{1}{p}}}+\|\widetilde{g}_u\|_{L_{\frac{1}{p}}}\right)^{\frac{1}{p}}\\
&=\left[\left(\int_{(K_1|u^\perp)\cap(K_2|u^\perp)}\widetilde{f}_u(x')^{\frac{1}{p}}\,dx'\right)^p+\left(\int_{(K_1|u^\perp)\cap(K_2|u^\perp)}\widetilde{g}_u(x')^{\frac{1}{p}}\,dx'\right)^p\right]^{\frac{1}{p}}\\
&=\bigg[a\left(\int_{(K_1|u^\perp)\cap(K_2|u^\perp)}\ell_{K_1}(u;x')\,dx'\right)^p+b\left(\int_{(K_1|u^\perp)\cap(K_2|u^\perp)}\ell_{K_2}(u;x')\,dx'\right)^p\bigg]^{\frac{1}{p}}\\
&=\left(a\vol_n(\widetilde{K}_1^u)^p+b\vol_n(\widetilde{K}_2^u)^p\right)^{\frac{1}{p}}.
\end{align*}
Equality holds in Minkowski's integral inequality if and only if $\widetilde{f}_u$ and $\widetilde{g}_u$ are proportional, i.e., if and only if $a\ell_{K_1}(u;x')^p=cb\ell_{K_2}(u;x')^p$ for all $x'\in (K_1|u^\perp)\cap (K_2|u^\perp)$ and some constant $c>0$. This means that 
\[
\ell_{K_1}(u;x')=\frac{cb}{a}\odot_{u,p}\ell_{K_2}(u;x')
\]
for all $x'\in(K_1|u^\perp)\cap(K_2|u^\perp)$ and some $c>0$, which holds if and only if $\widetilde{K}_1^u=\frac{cb}{a}\odot_{u,p}\widetilde{K}_2^u+tu$ for some $c>0$ and $t\in\R$.

    (ii) For $p>1$, the same proof above holds with the inequalities reversed, and the same equality conditions hold there as well.
    
    (iii) The case $p=1$ follows from the definition of the $L_1$ chord sum and Cavalieri's principle. 
\end{proof} 

The next result follows immediately from Theorem \ref{BMI-Lp-fiber-new}.

\bc\label{BMI-Lp-fiber-new-cor}
    Let $K_1,K_2\in\cK^n$ and $u\in\Sp$, and suppose that $K_1|u^\perp=K_2|u^\perp$.
    \begin{itemize}
        \item[(i)] If $p\in(0,1)$, then for all $a,b>0$, we have
    \begin{equation}\label{BMI-1}
        \vol_n(a\odot_{u,p} K_1\oplus_{u,p}b\odot_{u,p} K_2)^p \leq a\vol_n(K_1)^p+b\vol_n(K_2)^p
            \end{equation}
         with equality  if and only if $K_1=\frac{cb}{a}\odot_{u,p}K_2+tu$ for some constant $c>0$ and $t\in\R$.

        \item[(ii)] If $p>1$, then for all $a,b>0$, we have
    \begin{equation*}
        \vol_n(a\odot_{u,p} K_1\oplus_{u,p}b\odot_{u,p} K_2)^p \geq a\vol_n(K_1)^p+b\vol_n(K_2)^p
            \end{equation*}
       with  equality  if and only if $K_1=\frac{cb}{a}\odot_{u,p}K_2+tu$ for some constant $c>0$ and $t\in\R$.

        \item[(iii)]     If $p=1$, then for all $a,b>0$, we have
      \begin{equation*}
        \vol_n(a\odot_u K_1\oplus_u b\odot_u K_2)= a\vol_n(K_1)+b\vol_n(K_2).
    \end{equation*}
    \end{itemize}
\ec

\begin{remark}
    If $p=0$, then for all $a,b>0$, we have
        \[
\vol_n(a\odot_{u,0} K_1\oplus_{u,0}b\odot_{u,0} K_2)\leq a\vol_n(\widetilde{K}_1^u)+b\vol_n(\widetilde{K}_2^u)
        \]
      with equality if and only if $\widetilde{K}_1^u=\widetilde{K}_2^u$. This follows along the same lines as the proof of Theorem \ref{BMI-Lp-fiber-new}(i), but now we instead use the AM-GM inequality and its equality conditions. 
\end{remark}

\begin{remark}\label{scaling-Lp-fiber-sum}
    Let $u\in\Sp$, $a,b>0$, $p\in\R$ and $K\in\mathcal{K}^n$. Then
    \begin{align*}
        \vol_n(a\odot_{u,p}K\oplus_{u,p}b\odot_{u,p}K) &= \int_{(a\odot_{u,p}K\oplus_{u,p}b\odot_{u,p}K)|u^\perp}\ell_{a\odot_{u,p}K\oplus_{u,p}b\odot_{u,p}K}(u;x')\,dx'\\
        &=\int_{K|u^\perp}\left(a\ell_K(u;x')^p+b\ell_K(u;x')^p\right)^{\frac{1}{p}}\,dx'\\
        &=(a+b)^{\frac{1}{p}}\vol_n(K).
    \end{align*}
\end{remark}

\subsection{The $L_p$ chord mixed surface area  and its Minkowski's first inequality}

For $u\in\Sp$, $p\in[-\infty,+\infty]$, and $K_1,K_2\in\cK^n$, we define the \emph{$L_p$ chord mixed  surface area} by 
		\begin{align}\label{Lpfibermixed-def}
			S_{u,p}^{\oplus}(K_1,K_2) &:=\lim_{\varepsilon \to 0^+} \frac{\vol_{n}(K_1\oplus_{u,p}\varepsilon\odot_{u,p} K_2) -\vol_{n}(K_1)}{\varepsilon}.
		\end{align}
  
The Minkowski's first inequality for the $L_p$ chord  sum reads as follows.
		
\bt\label{MFI-Lp-fiber}
		Let $K_1,K_2\in\cK^n$ be such that $K_1|u^\perp=K_2|u^\perp$ for some $u\in\Sp$. Let $p\in(0,1)$. Then 
	\begin{equation}\label{mixed-graph-vol-ineq-1}
			S_{u,p}^{\oplus}(K_1, K_2) \leq \frac{1}{p}\cdot\vol_n(K_1)^{1-p}\vol_{n}(K_2)^{p}.
			\end{equation}
   If $p>1$, then the inequality in \eqref{mixed-graph-vol-ineq-1} is reversed. Equality holds if and only if $K_1=c\odot_{u,p}K_2+tu$ for some $c>0$ and $t\in\R$. 
		\et
  
		\begin{proof}
We again adapt the proof of the classical $L_p$ Minkowski's inequality in \cite[Theorem 7.2]{Gardner-BMI}, now to the setting of the $L_p$ chord sum. Substituting $\varepsilon=t/(1-t)$ in \eqref{Lpfibermixed-def}, by Proposition \ref{proposition}(iv) and Remark \ref{scaling-Lp-fiber-sum},
\begin{align*}
    S_{u,p}^{\oplus}(K_1,K_2) &=\lim_{t\to 0^+}\frac{\vol_n\left(K_1\oplus_{u,p}\frac{t}{1-t}\odot_{u,p}K_2\right)-\vol_n(K_1)}{t/(1-t)}\\
    &=\lim_{t\to 0^+}\frac{\vol_n\left(\frac{1}{1-t}\odot_{u,p}\left[(1-t)\odot_{u,p}K_1\oplus_{u,p}t\odot_{u,p}K_2\right]\right)-\vol_n(K_1)}{t/(1-t)}\\
    &=\lim_{t\to 0^+}\frac{(1-t)^{-1/p}\vol_n\left((1-t)\odot_{u,p}K_1\oplus_{u,p}t\odot_{u,p}K_2\right)-\vol_n(K_1)}{t/(1-t)}\\
    &=\lim_{t\to 0^+}\frac{\vol_n\left((1-t)\odot_{u,p}K_1\oplus_{u,p}t\odot_{u,p}K_2\right)-(1-t)^{1/p}\vol_n(K_1)}{t(1-t)^{\frac{1-p}{p}}}\\
    &=\lim_{t\to 0^+}\frac{\vol_n\left((1-t)\odot_{u,p}K_1\oplus_{u,p}t\odot_{u,p}K_2\right)-\vol_n(K_1)}{t}+\lim_{t\to 0^+}\frac{\left[1-(1-t)^{1/p}\right]\vol_n(K_1)}{t}\\
    &=\lim_{t\to 0^+}\frac{\vol_n\left((1-t)\odot_{u,p}K_1\oplus_{u,p}t\odot_{u,p}K_2\right)-\vol_n(K_1)}{t}+\frac{\vol_n(K_1)}{p}.
\end{align*}
Now consider the function $f_p:[0,1]\to(0,\infty)$ defined by $f_p(t):=g_p(t)^p$, where $g_p(t):=\vol_n((1-t)\odot_{u,p}K_1\oplus_{u,p}t\odot_{u,p}K_2)$. The preceding computation shows that
\begin{align*}
    f_p'(0) &=\frac{d^+}{dt}[g_p(t)^p]_{t=0}\\
    &=p\vol_n(K_1)^{p-1}\cdot\lim_{t\to 0^+}\frac{\vol_n((1-t)\odot_{u,p}K_1\oplus_{u,p}t\odot_{u,p}K_2)-\vol_n(K_1)}{t}\\
    &=p\vol_n(K_1)^{p-1}\left[S_{u,p}^{\oplus}(K_1,K_2)-\frac{\vol_n(K_1)}{p}\right]\\
    &=\frac{pS_{u,p}^{\oplus}(K_1,K_2)-\vol_n(K_1)}{\vol_n(K_1)^{1-p}}.
\end{align*}
Thus, the desired inequality \eqref{mixed-graph-vol-ineq-1} is equivalent to $f_p'(0)\leq f_p(1)-f_p(0)$. Note that
\[
f_p(1)=\vol_n(0\odot_{u,p}K_1\oplus_{u,p}1\odot_{u,p}K_2)^p=\vol_n(S_u(K_2))^p=\vol_n(K_2)^p,
\]
and similarly $f_p(0)=\vol_n(K_1)^p$. Now Corollary \ref{BMI-Lp-fiber-new-cor} says that $f_p(t)$ is convex, and by Remark \ref{scaling-Lp-fiber-sum}, 
            \begin{align*}
                S_{u,p}^{\oplus}(K_1,K_1) &=\lim_{\varepsilon \to 0^+} \frac{\vol_{n}(K_1\oplus_{u,p}\varepsilon\odot_{u,p} K_1) -\vol_{n}(K_1)}{\varepsilon}\\
                &=\lim_{\varepsilon \to 0^+} \frac{\left[(1+\varepsilon)^{\frac{1}{p}}-1\right]\vol_{n}(K_1)}{\varepsilon}\\
                &=\frac{\vol_n(K_1)}{p}.
            \end{align*}
          The desired inequality \eqref{mixed-graph-vol-ineq-1} follows. 

          Now suppose that equality holds in \eqref{mixed-graph-vol-ineq-1}. Then by the preceding remarks, $f_p'(0)=f_p(1)-f_p(0)$. Since $f_p$ is convex, 
          \[
\forall t\in(0,1],\qquad \frac{f_p(t)-f_p(0)}{t}=f_p(1)-f_p(0).
          \]
          This is the equality in \eqref{BMI-1}. Therefore, the equality conditions in \eqref{mixed-graph-vol-ineq-1} are the same as those in \eqref{BMI-1}.
            
            The inequality in the case $p>1$ follows in the same way (with the same equality conditions), but now we use Corollary \ref{BMI-Lp-fiber-new-cor}(ii) instead. 
		\end{proof}

		\vskip 5mm \small
		\noindent {\bf Acknowledgments.} We are grateful to Dylan Langharst for the discussions on the topic of this paper. We also would like to thank the anonymous referee for the comments which helped improve the paper. This material is based upon work supported by the National Science Foundation under Grant No. DMS-1929284 while the authors were in residence at the Institute for Computational and Experimental Research in Mathematics in Providence, RI, during the Harmonic Analysis and Convexity program. The second named author gratefully acknowledges support from the Simons Foundation through a Simons Travel Grant---GR021794.
		
		\bibliographystyle{plain}
		\bibliography{main}

\begin{thebibliography}{10}

\bibitem{inf-proof}
Field {T}heory/{T}he real numbers.
\newblock \url{https://en.wikibooks.org/wiki/Field\_Theory/The\_real\_numbers},
  Nov. 6, 2017.
\newblock [Online; accessed 11-February-2025].

\bibitem{A1}
S.~Alesker.
\newblock Continuous rotation invariant valuations on convex sets.
\newblock {\em Annals of Mathematics}, 149:977--1005, 1999.

\bibitem{A2}
S.~Alesker.
\newblock Description of translation invariant valuations on convex sets with a
  solution of {P}. {M}c{M}ullen's conjecture.
\newblock {\em Geometric \& Functional Analysis GAFA}, 11:244--272, 2001.

\bibitem{symm-in-geom}
G.~Bianchi, R.~J. Gardner, and P.~Gronchi.
\newblock Symmetrization in geometry.
\newblock {\em Advances in Mathematics}, 306:51--88, 2017.

\bibitem{symm-chapter}
G.~Bianchi and P.~Gronchi.
\newblock Symmetrizations.
\newblock In A.~Colesanti and M.~Ludwig, editors, {\em Convex Geometry:
  Cetraro, Italy 2021}, volume 2332 of {\em Lectures Notes in Mathematics},
  chapter~5, pages 233--292. Springer, Cham, 1 edition, 2023.

\bibitem{Bl1}
W.~Blaschke.
\newblock {\em Vorlesungen {\"u}ber {D}ifferentialgeometrie II, {A}ffine
  {D}ifferentialgeometrie}.
\newblock Springer Verlag, Berlin, 1923.

\bibitem{Boe-2000}
K.~J. B\"or\"oczky.
\newblock Approximation of general smooth convex bodies.
\newblock {\em Advances in Mathematics}, 153(2):325--341, 2000.

\bibitem{BLYZ-2012}
K.~J. B\"or\"oczky, E.~Lutwak, D.~Yang, and G.~Zhang.
\newblock The log-{B}runn--{M}inkowski inequality.
\newblock {\em Advances in Mathematics}, 231:1974--1997, 2012.

\bibitem{BFR-book}
K.~J. B\"or\"oczky, J.~P.~G. Ramos, and A.~Figalli.
\newblock {\em Isoperimetric inequalities, Brunn--Minkowski theory and
  Minkowski type Monge--Amp\`{e}re equations on the sphere}.
\newblock 2025.

\bibitem{CaglarYe2016}
U.~Caglar and D.~Ye.
\newblock Affine isoperimetric inequalities in the functional
  {O}rlicz--{B}runn--{M}inkowski theory.
\newblock {\em Advances in Applied Mathematics}, 81:78--114, 2016.

\bibitem{campi2002lp}
S.~Campi and P.~Gronchi.
\newblock The ${L}_p$-{B}usemann--{P}etty centroid inequality.
\newblock {\em Advances in Mathematics}, 167(1):128--141, 2002.

\bibitem{ChenEtAl-2020}
S.~Chen, Y.~Huang, Q.~Li, and J.~Liu.
\newblock The ${L}_p$-{B}runn--{M}inkowski inequality for $p<1$.
\newblock {\em Advances in Mathematics}, 368:107166, 2020.

\bibitem{Firey}
W.~J. Firey.
\newblock $p$-means of convex bodies.
\newblock {\em Mathematica Scandinavica}, 10:17--24, 1962.

\bibitem{Gardner-BMI}
R.~J. Gardner.
\newblock The {B}runn--{M}inkowski {I}nequality.
\newblock {\em Bulletin of the American Mathematical Society}, 39(3):355--405,
  2002.

\bibitem{GHW-2013}
R.~J. Gardner, D.~Hug, and W.~Weil.
\newblock Operations between sets in geometry.
\newblock {\em Journal of the European Mathematical Society (JEMS)},
  15:2297--2352, 2013.

\bibitem{GaZ}
R.~J. Gardner and G.~Zhang.
\newblock Affine inequalities and radial mean bodies.
\newblock {\em American Journal of Mathematics}, 120:505--528, 1998.

\bibitem{glasgrub}
S.~Glasauer and P.~M. Gruber.
\newblock Asymptotic estimates for best and stepwise approximation of convex
  bodies {III}.
\newblock {\em Forum Mathematicum}, 9:383--404, 1997.

\bibitem{haberl2008lp}
C.~Haberl.
\newblock ${L}_p$ intersection bodies.
\newblock {\em Advances in Mathematics}, 217:2599--2624, 2008.

\bibitem{HaberlFranz2009}
C.~Haberl and F.~Schuster.
\newblock General ${L}_p$ affine isoperimetric inequalities.
\newblock {\em Journal of Differential Geometry}, 83:1--26, 2009.

\bibitem{HLPRY-2}
J.~Haddad, D.~Langharst, E.~Putterman, M.~Roysdon, and D.~Ye.
\newblock Affine {I}soperimetric {I}nequalities for {H}igher-{O}rder
  {P}rojection and {C}entroid {B}odies.
\newblock {\em arXiv:2304.07859}, 2025.

\bibitem{HLPRY-1}
J.~Haddad, D.~Langharst, E.~Putterman, M.~Roysdon, and D.~Ye.
\newblock Higher-order ${L}_p$ isoperimetric and {S}obolev inequalities.
\newblock {\em Journal of Functional Analysis}, 288(2):110722, 2025.

\bibitem{Hoehner2022}
S.~Hoehner.
\newblock Extremal general affine surface areas.
\newblock {\em Journal of Mathematical Analysis and Applications},
  505(2):Article 125506, 2022.

\bibitem{Hug1995}
D.~Hug.
\newblock Contributions to {A}ffine {S}urface {A}rea.
\newblock {\em Manuscripta Mathematica}, 91:283--301, 1996.

\bibitem{JW}
J.~Jenkinson and E.~Werner.
\newblock Relative entropies for convex bodies.
\newblock {\em Transactions of the American Mathematical Society},
  366:2889--2906, 2014.

\bibitem{Lei1986}
K.~Leichtweiss.
\newblock Zur {A}ffinoberfl\"{a}che konvexer {K}\"{o}rper.
\newblock {\em Manuscripta Mathematica}, 56:429--464, 1986.

\bibitem{ludwig1999}
M.~Ludwig.
\newblock Asymptotic approximation of smooth convex bodies by general
  polytopes.
\newblock {\em Mathematika}, 46:103--125, 1999.

\bibitem{Ludwig2009}
M.~Ludwig.
\newblock General affine surface areas.
\newblock {\em Advances in Mathematics}, 224:2346--2360, 2010.

\bibitem{LudR}
M.~Ludwig and M.~Reitzner.
\newblock A characterization of affine surface area.
\newblock {\em Advances in Mathematics}, 147:138--172, 1999.

\bibitem{lutwak}
E.~Lutwak.
\newblock The {B}runn--{M}inkowski--{F}irey theory {I}: {M}ixed volumes and the
  {M}inkowski problem.
\newblock {\em Journal of Differential Geometry}, 38:131--150, 1993.

\bibitem{Lu1}
E.~Lutwak.
\newblock The {B}runn--{M}inkowski--{F}irey theory {II}: {A}ffine and
  geominimal surface areas.
\newblock {\em Advances in Mathematics}, 118:244--294, 1996.

\bibitem{lutwak2000lp}
E.~Lutwak, D.~Yang, and G.~Zhang.
\newblock ${L}_p$ affine isoperimetric inequalities.
\newblock {\em Journal of Differential Geometry}, 56(1):111--132, 2000.

\bibitem{lutwak2002sharp}
E.~Lutwak, D.~Yang, and G.~Zhang.
\newblock Sharp affine ${L}_p$ {S}obolev inequalities.
\newblock {\em Journal of Differential Geometry}, 62(1):17--38, 2002.

\bibitem{lutwak2004lp}
E.~Lutwak, D.~Yang, and G.~Zhang.
\newblock On the ${L}_p$-{M}inkowski problem.
\newblock {\em Transactions of the American Mathematical Society},
  356(11):4359--4370, 2004.

\bibitem{LYZ}
E.~Lutwak, D.~Yang, and G.~Zhang.
\newblock The {B}runn--{M}inkowski--{F}irey inequality for non-convex sets.
\newblock {\em Advances in Applied Mathematics}, 48:407--413, 2019.

\bibitem{McMullen1999}
P.~McMullen.
\newblock New combinations of convex sets.
\newblock {\em Geometriae Dedicata}, 78(1):1--19, 1999.

\bibitem{MW2}
M.~Meyer and E.~Werner.
\newblock On the $p$-affine surface area.
\newblock {\em Advances in Mathematics}, 152:288--313, 2000.

\bibitem{PW}
G.~Paouris and E.~Werner.
\newblock Relative entropy of cone measures and ${L}_p$ centroid bodies.
\newblock {\em Proceedings of the London Mathematical Society}, 104:253--286,
  2012.

\bibitem{Putterman-2021}
E.~Putterman.
\newblock Equivalence of the local and global versions of the
  ${L}_p$-{B}runn--{M}inkowski inequality.
\newblock {\em Journal of Functional Analysis}, 280(9):108956, 2021.

\bibitem{Saraglou-2015}
C.~Saroglou.
\newblock Remarks on the conjectured log-{B}runn--{M}inkowski inequality.
\newblock {\em Geometriae Dedicata}, 177:353--365, 2015.

\bibitem{Sch}
R.~Schneider.
\newblock {\em Convex Bodies: The Brunn–Minkowski Theory}.
\newblock Encyclopedia of Mathematics and its Applications. Cambridge
  University Press, 2 edition, 2013.

\bibitem{SW5}
C.~Sch{\"u}tt and E.~Werner.
\newblock Surface bodies and $p$-affine surface area.
\newblock {\em Advances in Mathematics}, 187:98--145, 2004.

\bibitem{TW2004}
N.~S. Trudinger and X.~Wang.
\newblock The affine plateau problem.
\newblock {\em Journal of the American Mathematical Society}, 18(2):253--289,
  2004.

\bibitem{Ulivelli}
J.~Ulivelli.
\newblock Convergence properties of symmetrization processes.
\newblock {\em Advances in Applied Mathematics}, 146:102484, 2023.

\bibitem{WG2007}
W.~Weidong and L.~Gangsong.
\newblock ${L}_p$-mixed affine surface area.
\newblock {\em Journal of Mathematical Analysis and Applications},
  335:341--354, 2007.

\bibitem{Werner2012a}
E.~Werner.
\newblock R\'enyi divergence and ${L}_p$-affine surface area for convex bodies.
\newblock {\em Advances in Mathematics}, 230:1040--1059, 2012.

\bibitem{werner2008new}
E.~Werner and D.~Ye.
\newblock New ${L}_p$ affine isoperimetric inequalities.
\newblock {\em Advances in Mathematics}, 218(3):762--780, 2008.

\bibitem{WernerYe2010}
E.~Werner and D.~Ye.
\newblock Inequalities for mixed $p$-affine surface area.
\newblock {\em Mathematische Annalen}, 347(3):703--737, 2010.

\bibitem{Ye2012}
D.~Ye.
\newblock Inequalities for general mixed affine surface areas.
\newblock {\em Journal of the London Mathematical Society}, 85:101--120, 2012.

\bibitem{Ye2013}
D.~Ye.
\newblock On the monotone properties of general affine surfaces under the
  {S}teiner symmetrization.
\newblock {\em Indiana University Mathematics Journal}, 14:1--19, 2014.

\bibitem{Ye2015b}
D.~Ye.
\newblock New {O}rlicz affine isoperimetric inequalities.
\newblock {\em Journal of Mathematical Analysis and Applications},
  427:905--929, 2015.

\bibitem{Ye2016a}
D.~Ye.
\newblock Dual {O}rlicz--{B}runn--{M}inkowski theory: dual {O}rlicz
  ${L}_{\phi}$ affine and geominimal surface areas.
\newblock {\em Journal of Mathematical Analysis and Applications},
  443:352--371, 2016.

\bibitem{Zhang2007}
G.~Zhang.
\newblock New {A}ffine {I}soperimetric {I}nequalities.
\newblock {\em International Congress of Chinese Mathematicians (ICCM)},
  2:239--267, 2007.

\end{thebibliography}
		
		\vspace{3mm}
		\noindent Steven Hoehner, \ \ \ {\small \tt hoehnersd@longwood.edu}\\
		{ \em Department of Mathematics \& Computer Science,   Longwood University,
			Farmville, Virginia 23909 }

            \vskip 2mm
		
		\noindent Sudan Xing, \ \ \ {\small \tt sudanxing333@gmail.com}\\
		{ \em Department of Mathematics \& Statistics,   University of Arkansas at Little Rock,
			Little Rock, Arkansas 72204}
		

\end{document}